\documentclass{amsart}

\usepackage{amssymb}
\usepackage{enumerate}
\usepackage{graphicx}
\usepackage{vmargin}

\setpapersize{A4}
\setmarginsrb{2.7cm}{2.7cm}{2.7cm}{2.7cm}{10pt}{20pt}{10pt}{20pt}

\theoremstyle{plain}
\newtheorem{theorem}{Theorem}

\newtheorem{lemma}[theorem]{Lemma}

\theoremstyle{definition}
\newtheorem{remark}[theorem]{Remark}
\newtheorem{remarks}[theorem]{Remarks}

\newtheorem{example}[theorem]{Example}

\newtheorem{defn}[theorem]{Definition}
\newtheorem{defns}[theorem]{Definitions}

\newcommand{\Q}{{\mathbb{Q}}}
\newcommand{\N}{{\mathbb{N}}}
\newcommand{\R}{{\mathbb{R}}}
\newcommand{\T}{{\mathbb{T}}}
\newcommand{\Z}{{\mathbb{Z}}}

\newcommand{\cC}{{\mathcal{C}}}
\newcommand{\cD}{{\mathcal{D}}}
\newcommand{\cM}{{\mathcal{M}}}
\newcommand{\cR}{{\mathcal{R}}}

\newcommand{\hatH}{{\widehat{H}}}
\newcommand{\hphi}{{\hat{\phi}}}
\newcommand{\hf}{{\hat{f}}}

\newcommand{\bv}{{\mathbf{v}}}
\newcommand{\bal}{{\boldsymbol\alpha}}
\newcommand{\bbeta}{{\boldsymbol\beta}}

\newcommand{\barf}{\overline{f}}
\newcommand{\barF}{\overline{F}}

\newcommand{\ux}{\underline{x}}
\newcommand{\uy}{\underline{y}}

\newcommand{\tf}{\tilde{f}}
\newcommand{\tg}{\tilde{g}}
\newcommand{\tH}{\tilde{H}}
\newcommand{\tx}{\tilde{x}}
\newcommand{\tX}{\tilde{X}}
\newcommand{\tPhi}{{\tilde{\Phi}}}

\newcommand{\Conv}{\operatorname{Conv}}
\newcommand{\DF}{\operatorname{DF}}
\newcommand{\supp}{\operatorname{supp}}

\newcommand{\id}{\mathrm{id}}
\newcommand{\Max}{\mathrm{Max}}

\newcommand{\one}{D}
\newcommand{\two}{E}
\newcommand{\twor}{e}
\newcommand{\veps}{\varepsilon}
\newcommand{\oI}{{(\ell(t), r(t))}}
\newcommand{\sMZ}{\text{\rm\tiny MZ}}

\newcommand{\raw}{\rightarrow}
\def\MZ{{Misiurewicz and Ziemian}}
\newcommand{\rot}{\operatorname{rot}}
\def\hphi{{\hat{\phi}}}
\def\hkappa{{\hat{\kappa}}}
\newcommand{\dis}{\operatorname{dis}}
\def\be{{\mathbf e}}
\newcommand{\Bd}{\operatorname{Bd}}
\newcommand{\cI}{{\mathcal{I}}}
\newcommand{\cJ}{{\mathcal{J}}}
\newcommand{\us}{\underline{s}}
\newcommand{\uk}{\underline{k}}
\def\homeo{homeomorphism}
\def\homeos{homeomorphisms}
\newcommand{\Cl}{\operatorname{Cl}}
\newcommand{\Intt}{\operatorname{Int}}
\newcommand{\Ex}{\operatorname{Ex}}
\newcommand{\cB}{{\mathcal{B}}}
\newcommand{\cL}{{\mathcal{L}}}
\newcommand{\cH}{{\mathcal{H}}}
\def\bn{\mathbf{n}}
\newcommand{\goober}[1]{{#1-rotational set}} 
\newcommand{\plaingoober}{{rotational set}}
\newcommand{\vgoober}{\goober{$\bv$}} 
\newcommand{\algoober}{\goober{$\bal$}}
\newcommand{\bu}{{\mathbf u}}
\def\uw{\underline{w}}
\def\hGamma{\hat{\Gamma}}

\newcommand{\tz}{\tilde{z}}
\newcommand{\bp}{{\mathbf p}}
\newcommand{\hg}{{\hat{g}}}
\newcommand{\Min}{\operatorname{MS}}
\graphicspath{{./}{figures/}}
\begin{document}

\title{New Rotation Sets in a Family of Torus Homeomorphisms}
\date{October 2015}
\author{Philip Boyland}
\address{Department of
    Mathematics\\University of Florida\\372 Little Hall\\Gainesville\\
    FL 32611-8105, USA}
\email{boyland@ufl.edu}
\author{Andr\'e de Carvalho}
\address{Departamento de
    Matem\'atica Aplicada\\ IME-USP\\ Rua Do Mat\~ao 1010\\ Cidade
    Universit\'aria\\ 05508-090 S\~ao Paulo SP\\ Brazil}
\email{andre@ime.usp.br}
\author{Toby Hall}
\address{Department of Mathematical Sciences\\ University of
    Liverpool\\ Liverpool L69 7ZL, UK}
\email{tobyhall@liv.ac.uk}

\thanks{ We would like to thank F\'abio Armando Tal and
Salvador Addas-Zanata for numerous fruitful conversations. 
The authors are grateful for the support of FAPESP grants
  2010/09667-0 and \mbox{2011/17581-0}. This research has also been supported
  in part by EU Marie-Curie IRSES Brazilian-European partnership in
  Dynamical Systems (FP7-PEOPLE-2012-IRSES 318999 BREUDS)}

\begin{abstract}
We construct a family $\{\Phi_t\}_{t\in[0,1]}$ of homeomorphisms of
the two-torus isotopic to the identity, for which all of the rotation
sets $\rho(\Phi_t)$ can be described explicitly. We
analyze the bifurcations and typical behavior of rotation sets in the
family, providing insight into the general questions of toral rotation
set bifurcations and prevalence.  We show that there is a full measure
subset of~$[0,1]$, consisting of infinitely many mutually disjoint
non-trivial closed intervals, on each of which the rotation set mode
locks to a constant polygon with rational vertices; that the generic
rotation set in the Hausdorff topology has infinitely many extreme
points, accumulating on a single totally irrational extreme point at
which there is a unique supporting line; and that, although $\rho(\Phi_t)$
varies continuously with~$t$, the set of extreme points of~$\rho(\Phi_t)$
does not.  The family also provides examples of rotation sets for
which an extreme point is not represented by any minimal invariant
set, or by any directional ergodic measure.
\end{abstract}  

\maketitle

%---------------------------------------------------------------

\section{Introduction}

In the theory of dynamics on manifolds, {\em rotation vectors} are
used to describe the asymptotic motion of orbits: the magnitude of the
rotation vector gives the speed of motion, and its direction gives the
homology class which best approximates the motion. Rotation vectors in
this form were introduced by Schwartzman~\cite{schwartzman} using
invariant measures. A topological version was given by
Fried~\cite{fried}, and an elegant synthesis was provided by
Mather~\cite{mather}. 

The set of all rotation vectors realized by the orbits of a particular
dynamical system is called its {\em rotation set}, and gives a
(perhaps coarse) invariant of the total dynamics. Given a class of
dynamical systems, there are four natural questions one can ask about
their rotation sets.

\medskip

\noindent\textbf{I. Shapes.} Which sets can be realized as rotation
sets?

\medskip

\noindent\textbf{II. Representatives.} How much of the dynamics is revealed by the
  rotation set? Are there good dynamical representatives of every
  vector in the rotation set?

\medskip

\noindent\textbf{III. Bifurcations.} How do rotation sets vary in parameterized
  families?

\medskip
\noindent\textbf{IV. Prevalence.} What does the typical rotation set look like?

\medskip

The answers to these questions are most completely understood for
homeomorphisms of the circle (the classical case studied by Poincar\'e
and Denjoy), for degree-one endomorphisms of the circle, and for
homeomorphisms of the annulus isotopic to the identity. In this paper
we study homeomorphisms $\Phi\colon\T^2\to\T^2$ of the two-dimensional
torus which are isotopic to the identity. Given such a homeomorphism,
fix a lift $\tPhi\colon\R^2\to\R^2$ to the universal cover. The
motion of orbits of~$\Phi$ is measured by the {\em displacement
  cocycle} $\dis\colon \T^2\times\Z\to\R^2$ given by
\[
\dis(z, r) =  \tPhi^r(\tz) - \tz,
\]
which is independent of the choice $\tz$ of lift of~$z$. The {\em
  rotation vector} of a point~$z\in\T^2$ is given by
\[
\rho(z) = \lim_{r\to\infty}\frac{\dis(z,r)}{r}\,\in\R^2
\]
when this limit exists. The {\em pointwise rotation set} of~$\Phi$ can
then be defined by
\[
\rho_p(\Phi) = \{\rho(z)\,:\, z\in\T^2,\,\rho(z)\text{ exists}\}.
\]
The effect of changing the lift $\tPhi$ of $\Phi$ is to translate
$\rho_p(\Phi)$ by an integer vector, but all of the torus
homeomorphisms we consider will have natural preferred lifts, and we
suppress this dependence.

The pointwise rotation set is difficult to work with {\em a priori},
and Misiurewicz and Ziemian~\cite{MZ1} introduced the now standard
definition of what we refer to as the {\em Misiurewicz-Ziemian
  rotation set}, which in most situations is easier to work with and has better
properties:

\begin{equation}
\label{eq:MZ-defn}
\rho_{\sMZ}(\Phi) = \{\bv\in\R^2\,:\, \frac{\dis(z_i,r_i)}{r_i} \to \bv
\text{ for some sequences $(z_i)$ in $\T^2$ and $(r_i)$ in~$\N$ with
  $r_i\to\infty$}\}. 
\end{equation}
For example, it is immediate from the definition that $\rho_{\sMZ}(\Phi)$ is a
compact subset of~$\R^2$.

Misiurewicz and Ziemian also proved that $\rho_{\sMZ}(\Phi)$ is convex,
giving rise to a basic trichotomy: $\rho_{\sMZ}(\Phi)$ is either a
point, or a line segment, or has interior. Much of the early work on
rotation sets focused on the third case, while in recent years there
has been substantial progress on the first two cases. In this paper we
consider only rotation sets $\rho_{\sMZ}(\Phi)$ with interior.

\medskip

Calculating the rotation set of a specific homeomorphism~$\Phi$ is
difficult in general. For this reason, most work has either
concentrated on general properties of rotation sets, or on the careful
construction of examples of homeomorphisms whose rotation sets have
certain properties. In this paper we give what we believe to be the
first construction of a nontrivial family $\{\Phi_t\}_{t\in[0,1]}$ of
homeomorphisms whose rotation sets can be described explicitly.
We classify and describe all of the (uncountably many) different rotation
sets $\rho_{\sMZ}(\Phi_{t})$ and their bifurcations with the parameter
$t$.  We are therefore able to give answers to the four questions above for
the rotation sets in this family. In particular, the family yields the
first new examples of rotation sets in the literature since the work
of Kwapisz in the 1990s. We prove that, in fact, these new rotation sets
are typical in the sense that they contain a residual set in the
collection of all rotation sets in the family with the Hausdorff
topology.

\medskip

While the construction of the family is carried out in such a way as
to make the calculation of rotation sets possible, it is not targeted
to produce any particular behavior. The phenomena which we describe
therefore occur naturally within the family. The systematic study of
parametrized families of maps has led to enormous progress in the
study of dynamical systems. The complete description given here of all
the rotation sets in our family provides valuable insights into the
possible structures and bifurcations of torus rotation sets and
motivates questions and conjectures about the answers to the four
questions in the general case: see Section~\ref{sec:questions}.

\medskip

We now give a summary of the main results of the paper, together with
a description of some relevant results of other authors.  In broad
outline the rotation sets in the family conform with the well known
behavior of the rotation numbers of generic families of circle
\homeos: for parameters that are buried points in a Cantor set
$\cB\subset [0,1]$, the rotation number is irrational, while in the
closure of each complementary gap of $\cB$ the rotation number mode
locks at a rational value. The analog of rational rotation number for
rotation sets is for the rotation set to be a polygon with rational
vertices, while the analog of irrational rotation number is for the
rotation set to have infinitely many extreme points, some of which are
irrational.

\medskip\medskip

\subsection*{Question I. Shapes}
In order to describe the types of rotation set realized by the family,
we need some definitions. An extreme point of a convex subset of the
plane is called a {\em vertex} if it has multiple supporting lines,
and a {\em smooth point} otherwise. A vertex is {\em polygonal} if it
has a neighborhood in the rotation set which is isometric to the neighborhood of a vertex
of a polygon. An irrational vector~$\bv=(v_1,v_2)\in\R^2$ is {\em
  planar totally irrational} if $v_1$, $v_2$, and $1$ are rationally
independent (i.e. if translation by~$\bv$ induces a minimal
homeomorphism of the torus), and {\em partially irrational} otherwise.

There are three types of rotation set~$\rho_{\sMZ}(\Phi_t)$
(see Theorem~\ref{thm:torusrot}):

\medskip

\begin{description}
\item[Rational regular] The rotation set is a convex polygon with
  rational vertices (Figure~\ref{fig:rat-rotset}).
\item[Irrational regular] The rotation set has infinitely many
  rational polygonal vertices, which accumulate on a single irrational
  extreme point (Figure~\ref{fig:reg-rotset}). This irrational extreme
  point can be either a vertex or a smooth point, and can be either
  partially or totally irrational.
\item[Irrational exceptional] The rotation set has infinitely many
  rational polygonal vertices, which accumulate on two irrational
  extreme points (Figure~\ref{fig:ex-rotset}). 
 The irrational extreme points are the endpoints of an {\em
    exceptional interval} in the boundary of $\rho_{\sMZ}(\Phi_t)$, which has
  the property that, for all~$s$, it is either contained in, or disjoint from,
  $\rho_{\sMZ}(\Phi_s)$.
\end{description}
Polygons with rational vertices are the best understood
type of rotation set. Kwapisz proved~\cite{kwap1} that every rational
polygon in the plane can be realized as $\rho_{\sMZ}(\Phi)$ for some
$C^\infty$-diffeomorphism $\Phi\colon\T^2\to\T^2$.

The first example of a rotation set having an irrational extreme point
was also provided by Kwapisz~\cite{kwap2}: he constructed a
$C^1$-diffeomorphism whose rotation set has infinitely many rational
polygonal vertices accumulating on two partially irrational
vertices. As far as we are aware, our family provides the first
examples of rotation sets with totally irrational extreme points in
the literature (such rotation sets are in fact generic in the family,
as discussed under question~IV below). Crovisier and le Roux (personal
communication)  have previously constructed such an
example starting, like Kwapisz's construction, with Denjoy examples
on the circle.

\subsection*{Question II. Representatives}
The simplest version of this question has an affirmative answer for
every homeomorphism in the family: for every
$\bv\in\rho_{\sMZ}(\Phi_t)$, there is some $z\in\T^2$ with rotation
vector~$\bv$, so that $\rho_p(\Phi_t) = \rho_{\sMZ}(\Phi_t)$ (see
Theorem~\ref{thm:torusrot}).

Given this, the next question is whether or not every
$\bv\in\rho_{\sMZ}(\Phi_t)$ is represented by an entire compact invariant
set, ideally one which looks like the invariant set of the rigid
rotation of the torus induced by translation by~$\bv$. Here the answer
is less straightforward, and we require some definitions.

A minimal set~$D$ for a torus homeomorphism~$\Phi$ is called a {\em
  $\bv$-minimal set} if every element of~$D$ has rotation
vector~$\bv$. For a rigid rotation by~$\bv$, we have that
 $\dis(z,r) - r\bv = 0$ for all $z\in
\T^2$ and $r\in\N$. If this quantity is uniformly bounded over all $z$
in an invariant subset $Z$ of $\T^2$ and all $r\in\N$, then~$Z$ is said
to have {\em bounded deviation}. A $\bv$-minimal set with bounded
deviation is called a {\em\vgoober}. J{\"a}ger~\cite{J1} showed that a
\vgoober\ is indeed dynamically similar to rigid rotation: if $\bv$ is
irrational, then a \vgoober\ is always semi-conjugate to rigid
translation on either the torus (if $\bv$ is totally irrational) or
the circle (if $\bv$ is partially irrational).  

When $\bv$ is rational, a theorem of Franks~\cite{franks} states that
there is a periodic point $z$ with $\rho(z) = \bv$: in particular, its
orbit is a \vgoober. It follows from a result of
Parwani~\cite{parwani} that this periodic orbit can be chosen to have
the same topological type as a periodic orbit of the rigid rotation
induced by translation by~$\bv$.

Misiurewicz and Ziemian~\cite{MZ2} show that every $\bv$ in the
interior of the rotation set of an arbitrary homeomorphism is
represented by a \vgoober; and that there exist homeomorphisms~$\Phi$
for which $\rho_{\sMZ}(\Phi)$ is a polygon with rational vertices,
with the property that some vectors $\bv$ on the boundary of
$\rho_{\sMZ}(\Phi)$ are not represented by any $\bv$-minimal set.

In view of these results, the only question remaining concerns the
existence of dynamical representatives of irrational points~$\bv$ in
the boundary of $\rho_{\sMZ}(\Phi_t)$. The answer to this question depends on
the type of the rotation set (see Theorem~\ref{thm:torus2}).

\medskip

\begin{description}
\item[Rational regular] Every $\bv\in\rho_{\sMZ}(\Phi_t)$ is represented by a
  \vgoober. 
\item[Irrational regular] Every $\bv\in\rho_{\sMZ}(\Phi_t)$ except perhaps
  for the irrational extreme point is represented by a \vgoober. The
  irrational extreme point is always represented by a uniquely ergodic
  $\bv$-minimal set, but this set sometimes does not have bounded deviation.
\item[Irrational exceptional] Every $\bv\in\rho_{\sMZ}(\Phi_t)$ except for
  elements of the exceptional interval~$P$ is represented by a
  \vgoober. There are no $\bv$-minimal sets for any $\bv\in P$:
  however, there is a minimal set~$D$ such that $\rho_{\sMZ}(D,\Phi_t) = P$,
  where $\rho_{\sMZ}(D,\Phi_t)$ is defined as in~(\ref{eq:MZ-defn}) but with the
  sequence $(z_i)$ in $D$. As a consequence, $\Phi_t|_D$ is not uniquely ergodic, and in fact it
has exactly two ergodic invariant Borel measures. Irrational
exceptional homeomorphisms therefore provide examples in which there is
an extreme point~$\bv$ of the rotation set with the property that the
support of its representing ergodic measure contains points whose
rotation vector differs from~$\bv$. Such measures are called {\em
  lost} in the terminology of Geller and Misiurewicz~\cite{gellmis}
(cf.\ \cite{jenk,jenk2}).
\end{description}

The relationship between these results and recent work of
Zanata~\cite{zanata} and Le Calvez \& Tal~\cite{tal} is also worth
noting: see Remark~\ref{rmk:compare}c).

\subsection*{Question III. Bifurcations}
The {\em bifurcation set}~$\cB\subset[0,1]$, on which $\rho_{\sMZ}(\Phi_t)$ is not
locally constant, is a Cantor set of Lebesgue measure zero. The set of
rational regular parameters is the union of the closures of the
complementary gaps of~$\cB$, with different gaps corresponding to
different rational polygonal rotation sets. Irrational parameters are
buried points of~$\cB$: both irrational regular and irrational exceptional parameters are
dense in~$\cB$.

A theorem of Misiurewicz and Ziemian~\cite{MZ2} guarantees that
$\rho_{\sMZ}(\Phi_t)$ varies continuously (in the Hausdorff topology)
with~$t$. Tal and Zanata pointed out that Hausdorff continuity of the
set of extreme points of $\rho_{\sMZ}(\Phi_t)$ is a stronger property, and
asked whether the family~$\{\Phi_t\}$ has this stronger property. It
does not: the map from~$t$ to the set of extreme points of
$\rho_{\sMZ}(\Phi_t)$ is discontinuous exactly at irrational exceptional
parameters, and at the right hand endpoints of the complementary gaps
of~$\cB$ (see Theorem~\ref{thm:torusrot}). 

The first example of discontinuity of the set of extreme points of a
rotation set with interior was constructed by Tal (personal
communication). 

\subsection*{Question IV. Prevalence}
From the point of view of the parameter~$t$, the typical rotation set
is a rational polygon: this is the rational regular case, which occurs
in the union of the closures of the complementary gaps of~$\cB$, that
is, on a full measure set which contains an open dense subset
of~$[0,1]$.  This is in accord with a result of Passeggi~\cite{pass}
which states that the $C^0$-typical torus \homeo\ has a rotation set
that is a (perhaps degenerate) rational polygon.

An alternative point of view on the relative abundance of the
various types of rotation set is provided by examining the collection
of all rotation sets in the family with the Hausdorff topology. This
space is homeomorphic to a compact interval~$\cR$. Each of the three
types of rotation set is dense in~$\cR$. However the typical rotation
set (in the sense that the collection of such rotation sets contains a
dense $G_\delta$ subset of $\cR$) is of irrational regular type,
having an irrational extreme point which is both totally irrational
and smooth (see Theorem~\ref{thm:torusrot}).

\subsection*{Outline of the paper}

The family $\{\Phi_t\}$ is constructed from a family~$\{f_t\}$ of
continuous self-maps of the figure eight space, in such a way that the
rotation sets of $\Phi_t$ and $f_t$ agree for all~$t$. The rotation
sets of the maps~$f_t$ can in turn be described in terms of {\em digit
  frequency sets} of associated symbolic $\beta$-shifts, which were
analysed in~\cite{beta}.

We will therefore study rotation sets in three different contexts:
torus homeomorphisms, maps of the figure eight space, and symbolic
$\beta$-shifts. In Section~\ref{sec:prelims} we briefly cover relevant
definitions and results from general rotation theory. Necessary
results from~\cite{beta} about digit frequency sets are summarized in
Section~\ref{sec:dig-freq}.

In Section~\ref{sec:8rot} the family~$\{f_t\}$ is constructed, and the
rotation sets $\rho_{\sMZ}(f_t)$ are calculated. Theorem~\ref{thm:eightrot}
is the main statement about the structure of these rotation sets (and
hence about the structure of the rotation sets $\rho_{\sMZ}(\Phi_t)=\rho_{\sMZ}(f_t)$).

In Section~\ref{sec:torusrot} we use a theorem from~\cite{param-fam} to
{\em unwrap} the family~$\{f_t\}$ to the family of torus
homeomorphisms~$\{\Phi_t\}$. Results about dynamical
representatives are contained in Section~\ref{sec:representatives-beta}
(dealing with symbolic $\beta$-shifts) and
Section~\ref{sec:representatives-toruseight} (dealing with the
families $\{f_t\}$ and $\{\Phi_t\}$). Finally, in
Section~\ref{sec:questions}, we pose some questions motivated by the
phenomena observed in the family.

\medskip

For each $n\ge 3$, similar techniques can be used to construct
families of homeomorphisms of the $n$-torus $\T^n$ whose rotation sets
behave analogously to those of the family~$\{\Phi_t\}$: see
Remark~\ref{rmk:higher-dimensions}.

\medskip

The rotation set $\rho_{\sMZ}(\Phi_t)$ can be calculated explicitly
for each value of the parameter~$t$, using the algorithm given
in~\cite{beta} for determining digit frequency sets. In the irrational
case, this means that the sequences of (rational) extreme points
around the boundary of $\rho_{\sMZ}(\Phi_t)$, moving either clockwise
or counterclockwise from the extreme point~$(0,0)$, can be listed as
far as computational accuracy permits. Figure~\ref{fig:rat-rotset}
depicts the rational regular rotation set at $t=3/4$, which is a
quadrilateral with vertices $(0,0)$, $(2/3,0)$, $(3/5,1/5)$, and
$(0,1/2)$ (the dotted lines indicate the rotation set
$\rho_{\sMZ}(\Phi_1)$, which has extreme points $(0,0)$, $(1,0)$ and
$(0,1/2)$). Figure~\ref{fig:reg-rotset} depicts an irrational regular
rotation set at $t\simeq 0.4093$, with a single limiting extreme point
which is smooth and totally irrational (the generic case). Finally,
Figure~\ref{fig:ex-rotset} depicts an irrational exceptional rotation
set at $t\simeq 0.0811$, which has two limiting irrational extreme
points bounding an exceptional interval.

\begin{figure}[htbp]
\includegraphics[width=0.6\textwidth]{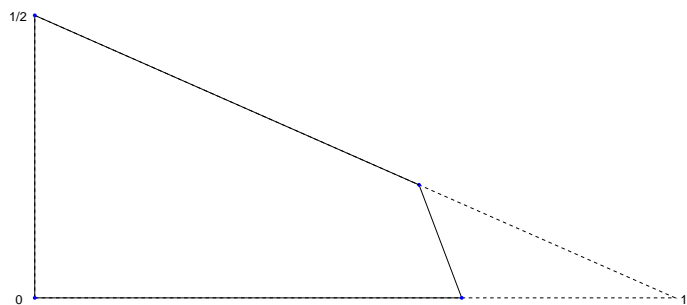}
\caption{The rational regular rotation set $\rho_{\sMZ}(\Phi_{3/4})$}
\label{fig:rat-rotset}
\end{figure}

\begin{figure}[htbp]
\includegraphics[width=0.6\textwidth]{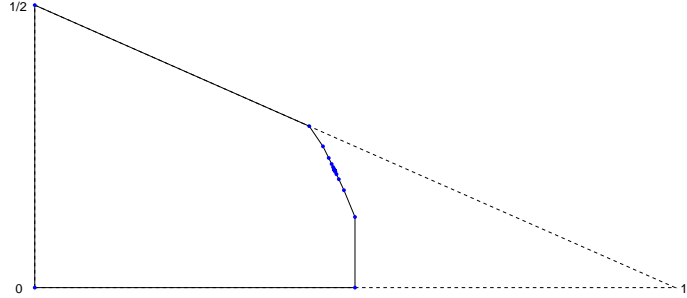}
\caption{An irrational regular rotation set at $t\simeq 0.4093$}
\label{fig:reg-rotset}
\end{figure}

\begin{figure}[htbp]
\includegraphics[width=0.6\textwidth]{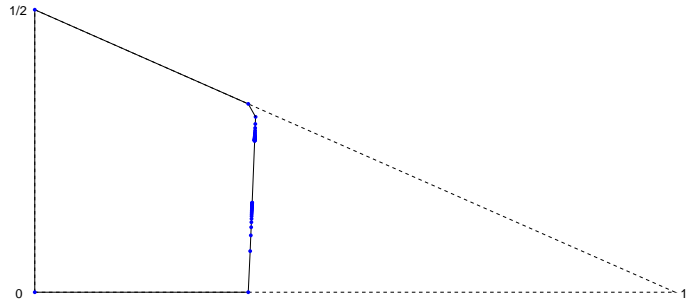}
\caption{An irrational exceptional rotation set at $t\simeq 0.0811$}
\label{fig:ex-rotset}
\end{figure}

\section{Preliminaries} 
\label{sec:prelims}

\subsection{Convex subsets of~$\R^n$}
We start by fixing some terminology and notation associated with
compact convex subsets of~$\R^n$, since there is considerable variance
in the literature.

$\cH(\R^n)$ will denote the set of all non-empty compact subsets
of~$\R^n$ with the Hausdorff topology. The convex hull, topological
boundary, and interior of an element $\Lambda$ of $\cH(\R^n)$ will be
denoted $\Conv(\Lambda)$, $\Bd(\Lambda)$, and $\Intt(\Lambda)$
respectively. 

A point $\bv$ of $\Bd(\Lambda)$ is said to be {\em extreme} if
$\Lambda-\{\bv\}$ is also a convex set; the set of all extreme points of
$\Lambda$ will be denoted $\Ex(\Lambda)$. A point $\bv\in\Ex(\Lambda)$
is a {\em polyhedral vertex} if $\Bd(\Lambda)$ is locally isometric to
the vertex of a polytope at $\bv$. A {\em limit extreme point} is a
limit point of $\Ex(\Lambda)$.

In the case $n=2$, a point $\bv\in\Ex(\Lambda)$ is said to be {\em
  smooth} if~$\Lambda$ has a unique supporting line at~$\bv$; and to
be a {\em vertex} otherwise.

\subsection{Rotation theory}
\label{sec:rot-theory}
In order to determine the rotation sets of the family $\{\Phi_t\}$ of
torus homeomorphisms, we will also need to study the collection of all
averages of observables on two other dynamical systems: symbolic
$\beta$-shifts, and continuous maps of the figure eight space. This study
was termed {\em rotation theory} by Misiurewicz~\cite{M}. In this
section we summarize the basic definitions and results which we will
need.

Let~$Z$ be a compact metric space, with dynamics given by a continuous
map $g\colon Z\to Z$. We will be interested in asymptotic averages of
a bounded and Borel measurable observable $\phi\colon Z\to\R^k$. To
this end, we define an associated {\em dynamical cocycle over~$g$},
denoted $\phi_g\colon Z\times\N\to\R^k$, by
\[
\phi_g(z, r) =\phi(z) + \phi(g(z)) + \dots + \phi(g^{r-1}(z)).
\]
The Birkhoff average of the observable, when it exists, is given by
\[
\hphi_g(z) = \lim_{r\raw\infty}\frac{\phi_g(z, r)}{r},
\]
and the {\em pointwise rotation set} of $g$ with respect to the observable $\phi$ is defined by
\[
\rot_p(Z, g, \phi) = \{\hphi_g(z)\,:\, z\in Z,\,\,\hphi_g(z)\text{ exists} \}.
\]

\MZ~\cite{MZ1} gave an alternative definition of the rotation set, which takes into
account asymptotic averages along subsequences:
 \[
\rot_{\sMZ}(Z, g, \phi) = \{\bv\in\R^k\,:\,
\frac{\phi_g(z_i, r_i)}{r_i} \to \bv \text{ for some sequences $(z_i)$
in~$Z$ and $(r_i)$ in $\N$ with $r_i\to\infty$}\}.
\]
It is evident that $\rot_{\sMZ}(Z, g, \phi)$ is  compact, whereas
$\rot_p(Z, g, \phi)$ need not be.

Let $\cM(g)$ denote the set of $g$-invariant Borel probability
measures on~$Z$, and $\cM_e(g)$ the subset of ergodic measures. Given
$\mu\in\cM(g)$, we write $\rot(\mu,g,\phi) = \int\phi\;d\mu$ for the
$\mu$-average of $\phi$. The {\em measure rotation set} and {\em
  ergodic measure rotation set} are defined by
\begin{align*}
\rot_m(Z, g, \phi) &=
 \{\rot(\mu, g, \phi)\,:\, \mu\in\cM(g)\}, \qquad\text{and}\\
\rot_{em}(Z, g, \phi) &=
 \{\rot(\mu, g, \phi)\,:\, \mu\in\cM_e(g)\}.
\end{align*}

There are inclusions
\begin{equation}
\label{eq:inclusions}
\rot_{em}(Z,g,\phi) \subset \rot_p(Z,g,\phi) \subset \rot_{\sMZ}(Z,g,\phi),
\end{equation}
the former coming from the pointwise ergodic theorem. We shall also
use that
\begin{equation}
\label{eq:m-conv-em}
\rot_m(Z,g,\phi) = \Conv\left(\rot_{em}(Z,g,\phi)\right).
\end{equation}

The proof of the following lemma, which relates the different
definitions of the rotation set in the case of continuous observables,
uses arguments and techniques from~\cite{MZ1} which carry over
without substantial change to the more general context considered here.
Brief details are provided for the reader's convenience. We write $R(g)$ for
the set of recurrent points of~$g$: that is, the set of points $z\in
Z$ with the property that $g^{r_i}(z)\to z$ for some sequence
$r_i\to\infty$.

\begin{lemma}[Misiurewicz and Ziemian]
\label{lem:MZ} 
Let $Z$ be a compact metric space, and suppose that $g\colon Z\to Z$
and $\phi\colon Z\to\R^k$ are continuous.
\begin{enumerate}[a)]
\item For each extreme point $\bv$ of
	$\Conv(\rot_{\sMZ}(Z, g, \phi))$, there is some
  $\mu\in\cM_e(g)$ with $\bv=\rot(\mu,g,\phi)$. In particular, there
  is a point $z\in R(g)$ with $\hphi_g(z)=\bv$.

\item If $\rot_{p}(Z, g, \phi)$ is convex, then
\[
 \rot_{m}(Z, g, \phi)  = \rot_{p}(Z, g, \phi) =
 \rot_{\sMZ}(Z, g, \phi).
\]

\item 
\[
\Conv(\rot_p(R(g), g, \phi)) = \Conv(\rot_{\sMZ}(Z,g,\phi)).
\]
In particular, if $W$ is a $g$-invariant subset of $Z$ containing
$R(g)$, and $\rot_p(W, g, \phi)$ is convex, then
\[ \rot_p(W, g, \phi) = 
\rot_{\sMZ}(Z, g, \phi).\]
\end{enumerate}
\end{lemma}
\begin{proof}
The first statement of~a) is proved in exactly the same way as
Theorem~2.4 of~\cite{MZ1}, and the second statement follows from the
pointwise ergodic theorem and Poincar\'e recurrence.

Part~b) follows from the observation that, if $\rot_p(Z,g,\phi)$ is convex, then
\[
\rot_m(Z,g,\phi) \subset \rot_p(Z,g,\phi) \subset \rot_{\sMZ}(Z,g,\phi) \subset
\Conv(\rot_{\sMZ}(Z,g,\phi)) \subset \rot_m(Z,g,\phi),
\]
where the first inclusion comes from taking convex hulls
in~(\ref{eq:inclusions}) and using~(\ref{eq:m-conv-em}) and the
convexity of~$\rot_p(Z,g,\phi)$; the second comes
from~(\ref{eq:inclusions}); the third is trivial; and the last comes
from part~a), the convexity of $\rot_m(Z,g,\phi)$, and the fact that,
since $\rot_{\sMZ}(Z,g,\phi)$ is a compact subset of~$\R^k$,
$\Conv(\rot_{\sMZ}(Z,g,\phi))$ is equal to the convex hull of its
extreme points.

For the first statement of~c), we have that
$\Conv(\rot_{\sMZ}(Z,g,\phi)) \subset \Conv(\rot_p(R(g),g,\phi))$ by
part~a). The reverse inclusion holds because $\rot_p(R(g),g,\phi)
\subset \rot_p(Z,g,\phi) \subset\rot_{\sMZ}(Z,g,\phi)$. In particular,
if $R(g)\subset W\subset Z$ and $\rot_p(W,g,\phi)$ is convex, then 
\[
\Conv(\rot_{\sMZ}(Z,g,\phi)) = \Conv(\rot_p(R(g),g,\phi)) \subset
\Conv(\rot_p(W,g,\phi)) = \rot_p(W,g,\phi) \subset \rot_{\sMZ}(Z,g,\phi),
\]
so that $\rot_{\sMZ}(Z,g,\phi)$ is also convex, and the second statement
follows.

\end{proof}

\medskip\medskip

Two observables $\phi,\psi\colon Z\to\R^k$ (and their
cocycles $\phi_g$ and $\psi_g$) are said to be {\em cohomologous with
  respect to~$g$} if there is a bounded measurable function $b\colon
Z\to\R^k$ satisfying
\begin{equation}
\label{eq:cohom1}
\phi(z)-\psi(z) = b(g(z))-b(z)\qquad\text{for all $z\in Z$}
\end{equation}
(in other contexts additional regularity conditions
are imposed on~$b$, but here boundedness suffices). This is equivalent
(see Theorem~2.9.3 of~\cite{HK}) to the existence of a constant~$C$ with
\begin{equation}
\label{eq:cohom2}
||\phi_g(z,r)-\psi_g(z,r)|| < C\qquad\text{for all $z\in Z$ and $r\in\N$}.
\end{equation}

If $\phi$ and $\psi$ are cohomologous then it is immediate
from~(\ref{eq:cohom1}) that $\rot_m(Z,g,\phi) = \rot_m(Z,g,\psi)$; and
from~(\ref{eq:cohom2}) that $\rot_p(Z,g,\phi) = \rot_p(Z,g,\psi)$ and
$\rot_{\sMZ}(Z,g,\phi)=\rot_{\sMZ}(Z,g,\psi)$.

\medskip\medskip

Suppose that $g\colon Z\to Z$ is semi-conjugate to $f\colon Y\to Y$ by
the surjective function $h\colon Z\to Y$, so that $h\circ g = f\circ
h$. Then any observable $\phi\colon Y\to\R^k$ on~$Y$ can be pulled
back to the observable $\psi = \phi\circ h$ on~$Z$, and
$\phi_f(h(z),r) = \psi_g(z,r)$ for all $z\in Z$ and
$r\in\N$. Since~$h$ is surjective, it follows that
\begin{equation}
\label{eq:semiconj}
\rot_p(Z,g,\psi) = \rot_p(Y,f,\phi) \qquad\text{ and }\qquad
\rot_{\sMZ}(Z,g,\psi) = \rot_{\sMZ}(Y,f,\phi).
\end{equation}

\medskip\medskip

We will need to understand how rotation sets transform under a very
simple example of an induced (or return) map. Suppose that $W\subset
Z$, and that there is a natural number~$K$ with the property that, for
every $z\in Z$, there is some $r$ with $1\le r\le K$ such that
$g^r(z)\in W$. Define an observable $N\colon W\to
\R$ taking values in $\{1,2,\ldots,K\}$, by 
\[
N(w) = \min\{r\ge 1\,:\,g^r(w)\in W\},
\]
and let the return map $R\colon W\to W$ be given by $R(w) =
g^{N(w)}(w)$.

Given an observable $\phi\colon Z\to\R^k$, define a corresponding
observable $\Phi\colon W\to\R^k$ by $\Phi(w) = \phi_g(w, N(w))$. Then
for all $w\in W$ and $r\in\N$ we have $ \Phi_R(w,r) =
\phi_g(w,N_R(w,r))$.

Since~$N$ is bounded, it follows that the Birkhoff average
$\hphi_g(w)$ exists if and only if the limit $\lim_{n\to\infty}
\Phi_R(w,n)/N_R(w,n)$ exists, and in this case the two are equal.
Since, moreover, the $g$-orbit of every point of~$Z$ enters~$W$, we
have
\begin{equation}
\label{eq:induce}
\rot_p(Z,g,\phi) = \left\{\lim_{r\to\infty}
\frac{\Phi_R(w,r)}{N_R(w,r)} \,:\, w\in W, \text{ the limit exists}\right\}.
\end{equation}
 The following lemma uses this to calculate $\rot_p(Z,g,\phi)$
 explicitly in the simple case of interest here.

\begin{lemma}
\label{lem:induce}
In the situation above, suppose that there is an observable
$\beta\colon W\to \Delta \subset\R^\ell$ for some compact convex
set~$\Delta$; and that there are linear maps $L\colon\R^\ell\to\R^k$
and $M\colon \R^\ell\to\R$ such that $\Phi = L\circ\beta$ and $N =
M\circ\beta$. Let $Q\colon\R^\ell\to\R^k$ be given by $Q(\bv) =
L(\bv)/M(\bv)$, and suppose that $Q|_\Delta$ is injective. Then
\[
\rot_p(Z,g,\phi) = Q\,(\rot_p(W,R,\beta)).
\]
\end{lemma}
\begin{proof}
By linearity of~$L$ and~$M$ we have $\Phi_R(w,r) = L(\beta_R(w,r))$
and $N_R(w,r) = M(\beta_R(w,r))$ for all~$w\in W$
and~$r\in\N$. Therefore, by~(\ref{eq:induce}), $\rot_p(Z,g,\phi)$ is
the set of all limits of the form $\lim_{r\to\infty} Q(\beta_R(w,r))$;
or, equivalently, since~$Q$ is homogeneous of degree~0, the set of all
limits of the form $\lim_{r\to\infty} Q(\beta_R(w,r)/r)$.  On the
other hand, $\rot_p(W,R,\beta) \subset \Delta$ is the set of all
limits of the form $\lim_{r\to\infty} \beta_R(w,r)/r$. Since
$Q|_\Delta$ is injective and continuous, it is a homeomorphism onto
its image, and the result follows.
\end{proof}

\section{Digit frequency sets of $\beta$-shifts}
\label{sec:dig-freq}

In this section we state some results from~\cite{beta} about the
possible frequencies of symbols (or {\em digits}) which can arise for
elements of symbolic $\beta$-shifts.

Let $\Sigma^+=\{0,1,2\}^\N$ be the one-sided sequence space over the
digits $0$, $1$, and $2$, ordered
lexicographically and endowed with the product topology; and let
$\sigma\colon\Sigma^+\to\Sigma^+$ be the shift map.

An element $\uw$ of $\Sigma^+$ is said to be {\em maximal} if $w_0 =
2$ and $\sigma^r(\uw)\le \uw$ for all $r\ge 0$.  Let
$\Max\subset\Sigma^+$ denote the set of all maximal
sequences. For each $\uw\in\Max$, the {\em symbolic $\beta$-shift}
associated to~$\uw$ is the subshift $\sigma\colon
B(\uw)\to B(\uw)$, where
\begin{equation}
\label{eq:betadef}
B(\uw) = \{\us\in\Sigma^+\,:\,\sigma^r(\us)\le \uw \text{ for all }r\ge
0\}.
\end{equation}

The (continuous) observable of interest is $\kappa:\Sigma^+ \raw
\R^3$, defined by $\kappa(\uw) = \be_{w_0}$ where $\{\be_0, \be_1,
\be_2\}$ is the standard basis of $\R^3$.  Therefore
$\hkappa_\sigma(\uw)\in\Delta$ gives the asymptotic frequency of the
digits in $\uw$, if it exists. Here
\[\Delta = \{\bal\in\R^3_{\ge 0}\,:\,\sum_{i=0}^2 \alpha_i=1\}\]
is the standard 2-simplex.  

The collection of digit frequencies realized in the symbolic
$\beta$-shift $B(\uw)$ is
\begin{equation*}
\DF(\uw) := \rot_p(B(\uw), \sigma, \kappa) \subset \Delta.
\end{equation*} 
We write
\[
\cD = \{\DF(\uw)\,:\,\uw\in\Max\} \subset \cH(\Delta)
\]
for the set of all digit frequency sets, equipped (anticipating
Theorem~\ref{thm:digit-freq}a) below) with the Hausdorff topology and
ordered by inclusion.

A vector $\bv\in\Delta$ is called {\em irrational} if
$\bv\not\in\Q^3$; and is called {\em totally irrational} if there is
no non-zero $\bn\in\Z^3$ with $\bn\cdot\bv = 0$.

\medskip

The following theorem is a summary of results from~\cite{beta} (see
Corollary~17 and Theorems~27, 33, 37, 38, 51, and~54 of that
paper). 

\begin{theorem}\mbox{}
\label{thm:digit-freq}
\begin{enumerate}[a)]
\item $\DF(\uw)$ is a compact, convex subset of $\Delta$ for all
  $\uw\in\Max$.
\item The map $\DF:\Max\raw \cD$ is continuous and non-decreasing.
\item There is a partition $\Max = M_1 \sqcup M_2 \sqcup M_3$ with
  the following properties:
\begin{enumerate}[i)]
\item If $\uw\in M_1$, then $\DF(\uw)$ is a polygon with rational
  vertices and there is an interval $I_{\uw}=[u(\uw), v(\uw)]\subset M_1$
  in $\Max$, with $\DF(\uw)=\DF(\uw')$ if and only if
  $\uw'\in I_{\uw}$.
\item If $\uw\in M_2$, then $\Ex(\DF(\uw))$ consists of infinitely many
  rational polyhedral vertices together with a single irrational limit
  extreme point.
\item If $\uw\in M_3$, then $\Ex(\DF(\uw))$ consists of infinitely many
  rational polyhedral vertices together with two irrational limit
  extreme points, which are endpoints of a line segment in $\Bd(\DF(\uw))$.
\item Each of $M_2$ and $M_3$ is uncountable,
and if $\uw\in M_2\cup M_3$ then $\uw$ is not eventually
 periodic, and $\DF(\uw)\not=\DF(\uw')$ for all $\uw'\not=\uw$.
\item The {\em bifurcation set}
\[
M_2 \cup M_3 \cup \bigcup_{\uw\in M_1}\{u(\uw), v(\uw)\}
\]
is a Cantor set.
\end{enumerate}
 \item $\cD$ is order-preserving homeomorphic to a compact interval,
   and each $\DF(M_i)$ is dense in~$\cD$.
\item There is a dense $G_\delta$ subset of~$\cD$, contained in
  $\DF(M_2)$, consisting of digit frequency sets whose limit extreme
  point is smooth and totally irrational.
\item The map $\Ex\circ\DF:\Max\raw \cH(\Delta)$ is discontinuous at
  each point of~$M_3$ and at $v(\uw)$ for each $\uw\in M_1$; and is
  continuous elsewhere.
\end{enumerate}

\end{theorem}

\begin{remark}
$\DF(\uw) = \rot_p(B(\uw), \sigma, \kappa) = \rot_{\sMZ}(B(\uw),
  \sigma, \kappa) = \rot_m(B(\uw), \sigma, \kappa)$ by part~a) of the
  theorem and Lemma~\ref{lem:MZ}b).
\end{remark}

\section{Rotation sets of a family of maps of the figure eight space}
\label{sec:8rot}
\subsection{The family of maps $f_t\colon X\to X$}
\label{sec:mapdefs}
Let $X=S_1\vee S_2$ be a wedge of two oriented circles $S_1$
and~$S_2$, with respective lengths~5 and~3, meeting at a
vertex~$v$. We use the orientation to define an order on each of the
circles: if $x$ and $y$ belong to the same circle $S_i$, then we say
that $x\le y$ if the oriented arc of $S_i$ from $v$ to $y$
contains~$x$.

Subdivide the circles $S_1$ and $S_2$ into five and three oriented
compact subintervals (edges) of length~1, so that they can be written
as edge-paths (see Figure~\ref{fig:graph-map})
\begin{eqnarray*}
S_1 &=& C\,c\,\one\,\two\,\twor, \quad\text{ and}\\
S_2 &=& A\,B\,b
\end{eqnarray*}
(the motivation for this labelling is that the
images of edges will be orientation-preserving or
orientation-reversing according as they are denoted with upper or
lower case letters). Define a map $f\colon X\to X$ homotopic to the identity, which
expands each edge uniformly by a factor of either 5 or 3 (depending on
whether its
image is $S_1$ or $S_2$), with the oriented edge images given by
\begin{equation}
\label{eq:mapdef}
\begin{array}{llllll}
\text{Edges in $S_1$:\qquad} & f(C) = S_2, \quad & f(c) = S_2^{-1}, \quad & f(\one) = S_1, \quad &
f(\two) = S_1, \quad & f(\twor) = S_1^{-1},\\
\text{Edges in $S_2$:\qquad} & f(A) = S_2, \quad & f(B) = S_1, \quad & f(b) = S_1^{-1},&&
\end{array}
\end{equation}
where $S_i^{-1}$ denotes the circle $S_i$ traversed with reversed
orientation. See the upper part of Figure~\ref{fig:graph-map}, in
which the circles $S_1$ and $S_2$ are drawn horizontally and
vertically respectively, and the images of each circle have been
separated for clarity. 

\begin{figure}[htbp]
\includegraphics[width=0.8\textwidth]{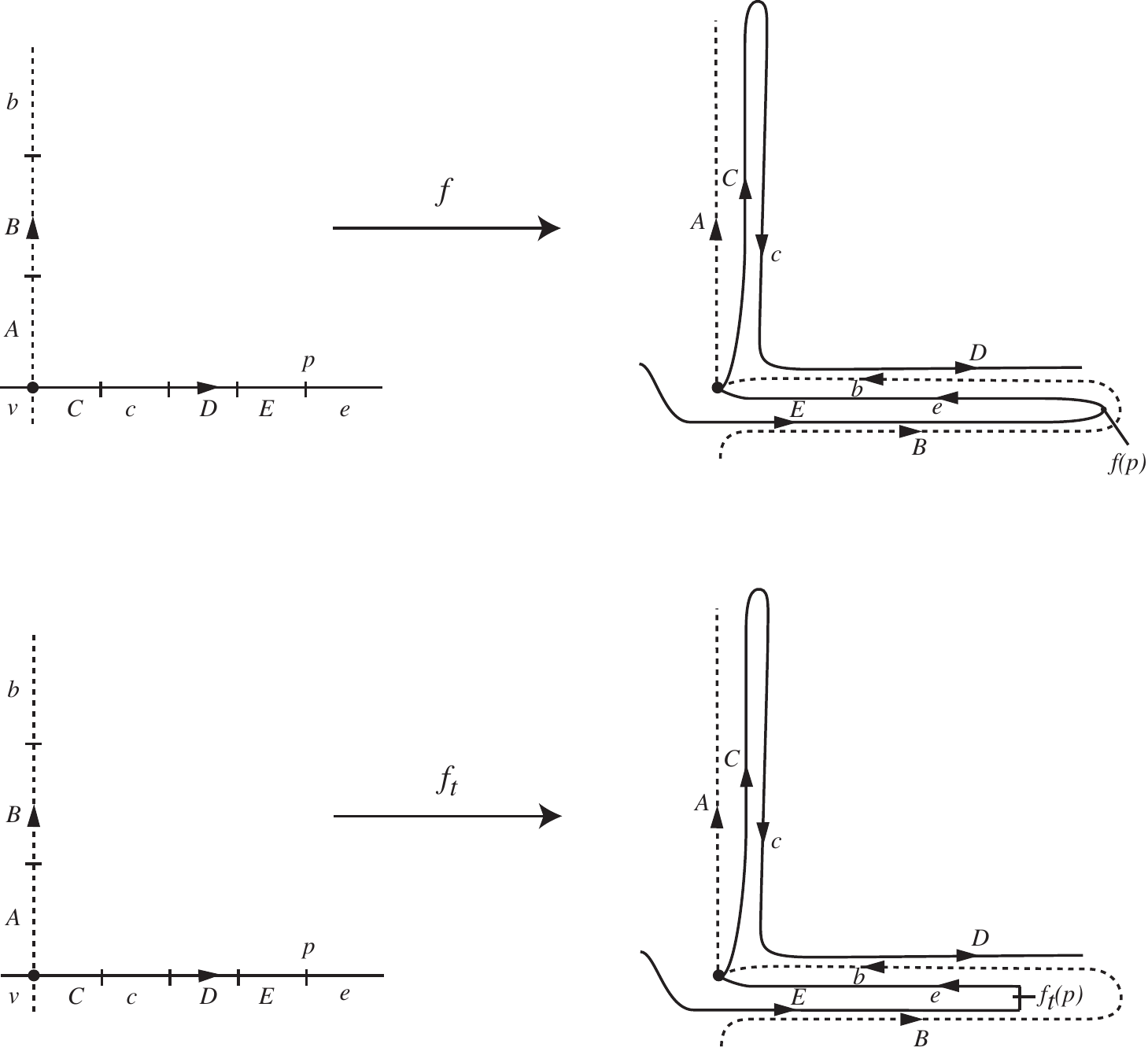}
\caption{The maps $f\colon X\to X$ and $f_t\colon X\to X$}
\label{fig:graph-map}
\end{figure}

Let~$p$ denote the common endpoint of the edges $\two$ and
$\twor$. The parameterized family of maps $f_t\colon X\to X$ is
defined by ``cutting off'' the tip of the transition
$\two\,\twor\mapsto S_1\,S_1^{-1}$, as depicted in the lower part of
Figure~\ref{fig:graph-map}: it is an analog of the stunted tent family
on the interval which, in contrast to the standard tent family, is
full. Define $\ell,r\colon [0,1]\to X$ so that $\ell(t)$ and $r(t)$
are the points of $\two$ and $\twor$ respectively which are distance
$(5-3t)/10$ from~$p$: thus $f(\ell(t))=f(r(t))$ for all~$t$$;
f(\ell(0))$ is the fixed point of~$f$ in $\one$; and $f(\ell(1)) =
p$. Then for each $t\in[0,1]$, let $I_t=[\ell(t),r(t)]\subset
S_1$. The maps $f_t\colon X\to X$ are defined for $t\in[0,1]$ by
\[
f_t(x) = \left\{
\begin{array}{ll}
f(x) \quad & \text{ if }x\not\in I_t,\\
f(\ell(t)) \quad & \text{ if }x\in I_t.
\end{array}
\right.
\]

\medskip

Points of~$X$ can be assigned (perhaps multiple) itineraries
under~$f_t$ belonging to the set
\[
\Sigma_8^+ = \{\uk\in\{A,B,b,C,c,\one,\two,\twor\}^\N\,:\, \text{ for all 
}r\ge 0,\,\,k_{r+1}\in\{A,B,b\}
\text{ if and only if }k_r\in\{A,C,c\}\}
\]
in the standard way: an element $\uk$ of $\Sigma_8^+$ is an
$f_t$-itinerary of $x\in X$ if and only if $f_t^r(x)\in k_r$ for each
$r\in\N$. (The condition that $k_{r+1}\in\{A,B,b\}$ if and only if
$k_r\in\{A,C,c\}$ comes from the transitions specified
in~(\ref{eq:mapdef}).) A point~$x$ has more than one $f_t$-itinerary
if and only if its orbit under~$f_t$ passes through an endpoint of one
of the defining intervals.

Every $\uk\in\Sigma_8^+$ is an $f$-itinerary of a unique~$x\in X$, by the
standard argument: the sets of points of $X$ which have an
$f$-itinerary agreeing with $\uk$ to $r$ symbols form a decreasing
sequence of non-empty compact sets, whose diameter goes to zero
because of the expansion of~$f$.

\subsection{Rotation sets of the maps~$f_t$}
\label{sec:rot-sets-8}
Let $p\colon \tX\to X$ be the universal Abelian cover of~$X$, which we
represent as $\tX = (\R\times\Z)\cup(\Z\times\R) \subset\R^2$, with
the coordinates chosen so that $p(\R\times\Z) = S_1$ and
$p(\Z\times\R) = S_2$. For each~$t\in[0,1]$, let $\tf_t\colon
\tX\to\tX$ be the unique lift of $f_t$ which fixes each integer
lattice point.

The rotation set of $f_t$ is defined using the continuous observable
$\gamma^t\colon X\to \R^2$ defined by $\gamma^t(x) = \tf_t(\tx) - \tx$,
where $\tx\in\tX$ is an arbitrary lift of $x\in X$: we will study the
sets
\[
\rho_8(t) = \rot_p(X, f_t, \gamma^t)
\]
(which will be shown in Theorem~\ref{thm:eightrot} below
to be equal to $\rot_{\sMZ}(X, f_t, \gamma^t)$ and to $\rot_m(X,f_t,\gamma^t)$).

In order to calculate rotation sets using symbolic techniques, it is
convenient to use a discrete version $\Gamma\colon
X\to\Z^2\subset\R^2$ of~$\gamma^t$, defined (except at preimages of~$v$)
by
\[
\begin{array}{ll}
\Gamma(x) = (0,0) \quad & \text{ for }x\in A\cup C\cup c\cup \one,\\
\Gamma(x) = (0,1) \quad & \text{ for }x\in B \cup b, \text{ \ and}\\
\Gamma(x) = (1,0) \quad & \text{ for }x\in \two\cup \twor.
\end{array}
\]

For each $(m,n)\in\Z^2$, let $D(m,n)\subset\tX$ be the fundamental
domain consisting of the points with coordinates $(m+x,n)$ for
$x\in[0,1)$ and $(m,n+y)$ for $y\in[0,1)$. Then (again with the
    exception of preimages of~$v$, for which the rotation vector
    is $(0,0)$), if $\tx\in D(m,n)$ then $\tf_t(\tx)\in
    D((m,n)+\Gamma(x))$, and hence $\tf_t^r(\tx)\in
    D((m,n)+\Gamma_{f_t}(x,r))$. It follows that $||\Gamma_{f_t}(x,r)
    - \gamma^t_{f_t}(x,r)||$ is uniformly bounded over $t\in[0,1]$,
    $x\in X$, and $r\in\N$. Therefore~$\gamma^t$ and $\Gamma$ are
    cohomologous with respect to~$f_t$, and in particular $\rho_8(t) = \rot_p(X, f_t,\Gamma)$. 

\subsection{Invariant subsets}
\label{sec:inv-sub}

In order to compute the rotation sets $\rho_8(t)$ we will make use of
successively smaller $f_t$-invariant subsets of~$X$ which carry the
entire rotation set, and we now introduce these subsets. For each
$t\in[0,1]$, write $o(x,f_t) =\{f_t^r(x)\,:\, r\ge 0\}$ for the orbit
of~$x$ under~$f_t$. Let
\[
W = \{x\in X\,:\, o(x,f) \subset B\cup C\cup \one\cup\two\},
\]
the set of points whose orbits do not enter the interiors of the
orientation reversing intervals or of the interval~$A$. For
each~$t\in[0,1]$ set
\[
X_t = \{x\in X\,:\, o(x,f_t) \subset X-\oI\},
\]
the maximal $f_t$-invariant set on which $f_t=f$; and define 
\[Y_t = W\cap X_t.\]
Then $W$, $X_t$, and $Y_t$ are compact subsets of~$X$;
and $f_t$ and $f$ are equal on $X_t$ and $Y_t$.

It is convenient to reduce the size of the intervals $B$, $C$,
and $\one$ so as to remove common endpoints. Let $B'$ and $D'$ be the
initial segments of $B$ and $D$ of length~$4/5$ (so $f(B')=f(D')
= C\cup \,c\,\cup \one\cup \two$), and let $C'$ be the segment of $C$ with
$f(C')=B'$. Clearly if $x\in W$ then $o(x,f) \subset B'\cup C'\cup
\one'\cup \two$.

\medskip

Since every point~$x$ of $Y_t \cap C'$ has $f(x)\in B'$ and $f^2(x)\in
Y_t\cap S_1$, we can study the dynamics of $f_t$ on~$Y_t$ using the
first return map $R_t\colon Z_t\to Z_t$ to the subset \[Z_t=Y_t\cap
S_1\] of $Y_t$.  Define $F\colon C'\cup \one'\cup \two\to X$ by $F =
f^2$ on $C'$ and $F = f$ on $\one'\cup \two$; then $R_t \colon Z_t\to
Z_t$ is given by $R_t = F|_{Z_t}$.

Define an itinerary map $h_t\colon Z_t\to\Sigma^+$ for $R_t$ by
\[
h_t(x)_r = \left\{
\begin{array}{ll}
0 & \text{ if }R_t^r(x)\in C',\\
1 & \text{ if }R_t^r(x)\in \one', \text{ and}\\
2 & \text{ if }R_t^r(x)\in \two.
\end{array}
\right.
\]
We emphasize again that, since $Y_t$ and $Z_t$ are subsets of~$X_t$,
we have $f_t=f$ on $Y_t$, $R_t = R_1 = F$ on $Z_t$, and $h_t=h_1$ on $Z_t$.

\medskip

Define
\[
\cC = \{t\in [0,1]\,:\, \ell(t) \in Z_t\},
\]
the set of parameters for which $\ell(t)\in W$ and $\ell(t)$ is the
greatest point of $S_1$ on its $f$-orbit. It follows that
\[
\cL := \ell(\cC) = \{x\in[\ell(0),\ell(1)]\,:\, o(x,f) \subset B\cup
C\cup\one\cup\two - (x,p]\}\subset Z_1.
\]
Since $\cL$ is evidently compact and $\ell$ is affine, $\cC$ is also compact. We
can therefore define a function $a\colon[0,1]\to\cC$ by
\[
a(t) = \max\{t'\in\cC\,:\, t'\le t\}.
\]

\medskip

Finally, let $K\colon \cC\to \Sigma^+$ be the ``kneading sequence'' map defined
by $K(t) = h_1(\ell(t))$.

\begin{lemma}\mbox{}
\label{lem:conjtobeta} 
\begin{enumerate}[a)]
\item  $Z_1$ is a Cantor set of Lebesgue measure zero, and $h_1$ is an
  order-preserving topological conjugacy between $R_1\colon Z_1\to
  Z_1$ and $\sigma\colon\Sigma^+\to\Sigma^+$.
\item $\cC$ has Lebesgue measure zero.
\item $K$ is an order-preserving homeomorphism onto its image
\[
K(\cC) =  \{\uw\in\Max\colon \uw \geq 2\overline{1}\},
\]
where the overbar denotes infinite repetition.
\item Let $t\in[0,1]$. Then $Z_t = Z_{a(t)}$, and $h_t$ is an
  order-preserving topological conjugacy between $R_t\colon Z_t\to
  Z_t$ and the symbolic $\beta$-shift $\sigma\colon B(K(a(t)))\to
  B(K(a(t)))$. 
\end{enumerate}
\end{lemma}

\begin{proof}
\begin{enumerate}[a)]
\item $F\colon C'\cup\one'\cup\two\to X$ maps each of the three
  disjoint compact intervals $C'$, $\one'$ and $\two$ affinely over
  all three of the intervals, with slope at least~5. Since $Z_1$ is
  the set of points whose $F$-orbits are contained in these
  intervals, and $R_1$ is the restriction of $F$ to $Z_1$, the
  result follows by standard arguments.
\item Since $\ell$ is an affine map and $\ell(\cC) \subset Z_1$, the
  result follows from~a).
\item $K=h_1\circ\ell$ is an order-preserving homeomorphism onto its
  image since both $h_1$ and $\ell$ are order-preserving, continuous,
  and injective, and $\cC$ is compact. $K(\cC)\subset\Max$, since if
  $\uw = K(t)$ for some $t\in\cC$ then we have $\sigma^r(\uw) =
  \sigma^r(h_1(\ell(t)) = h_1(R_1^r(\ell(t))) = h_1(R_t^r(\ell(t)))
  \le h_1(\ell(t)) = \uw$, using $\ell(t)\in Z_t$ for $t\in\cC$. Moreover
  $K(0) = 2\overline{1}$, since $f(\ell(0))$ is the fixed point of~$f$
  in~$\one$, so that $K(t)\ge 2\overline{1}$ for all $t\in\cC$.

It therefore only remains to show that every maximal sequence $\uw\ge
2\overline{1}$ is in the image of~$K$. Let~$\uk\in\Sigma_8^+$ be the
sequence obtained from~$\uw$ by the substitution $0\mapsto CB$,
$1\mapsto\one$, and $2\mapsto\two$, and let $x\in X$ be a point
with itinerary~$\uk$.  Since $K(0) = 2\overline{1} \le \uw\le
\overline{2} = K(1)$ we have $x\in [\ell(0), \ell(1)]$, so that there
is some $t\in[0,1]$ with $\ell(t) = x$. Then $x\in Z_t$ by maximality
of~$\uw$, so that $t\in\cC$ and $\uw = K(t)$.
\item Since $a(t)\le t$ we have $Z_{a(t)}\subset Z_t$. To show
  equality, suppose for a contradiction that there is some $x\in Z_t -
  Z_{a(t)}$. Since~$Z_t$ is compact and $R_t$-invariant, $y =
  \sup\{R_t^r(x)\,:\,r\ge 0\}$ is also an element of~$Z_t$, so that
  $y\le\ell(t)$. Moreover, $R_t^r(y)\le y$ for all $r\ge 0$ by
  continuity of $R_t$, so that $y\in\cL$ and hence $y=\ell(t')$ for
  some $t'\in\cC$. On the other hand, $y>\ell(a(t))$ since $x\not\in
  Z_{a(t)}$. Therefore $t'\in (a(t),t]\cap\cC$, contradicting the
definition of $a(t)$.

Since $h_t=h_1|_{Z_t}$ and $R_t = R_1|_{Z_t}$, it follows from
part~a) that $h_t$ conjugates $R_t\colon Z_t\to Z_t$ and $\sigma\colon
h_t(Z_t)\to h_t(Z_t)$.  We therefore need only show that $h_t(Z_t) =
B(K(a(t)))$. Now given $x\in Z_1$ we have
\begin{eqnarray*}
x\in Z_t & \iff & x\in Z_{a(t)} \\
& \iff & R_1^r(x) \le \ell(a(t)) \text{ for all $r\ge 0$}\\
& \iff & \sigma^r(h_1(x)) \le h_1(\ell(a(t))) \text{ for all $r\ge
  0$}\\
& \iff & h_1(x)\in B(h_1(\ell(a(t)))) = B(K(a(t))).
\end{eqnarray*}
Here the first equivalence is what we have just proved; the second is
the definition of $Z_{a(t)}$ (or, more particularly, of $X_{a(t)}$);
the third follows from part~a); and the fourth is by the
definition~(\ref{eq:betadef}) of $B(h_1(\ell(a(t))))$. Therefore
$h_1(Z_t) = B(K(a(t)))$, and so $h_t(Z_t) = B(K(a(t)))$ as required,
since $h_t$ and $h_1$ agree on $Z_t$.
\end{enumerate}
\end{proof}

\subsection{Calculation of the rotation sets $\rho_8(t)$}
In this section we apply Lemma~\ref{lem:conjtobeta}d) to relate the
rotation set $\rho_8(t)$ to the digit frequency set $\DF(K(a(t)))$,
and use this relationship together with Theorem~\ref{thm:digit-freq}
to describe the collection of rotation sets $\rho_8(t)$.

\begin{theorem}
\label{thm:rotequal} 
Let $t\in[0,1]$. Then $\rho_8(t)= \Pi(\DF(K(a(t))))$, where
$\Pi\colon\Delta\to\R^2$ is defined by
\[
\Pi(\alpha_0,\alpha_1,\alpha_2) = \left(
\frac{\alpha_2}{1+\alpha_0},\,
\frac{\alpha_0}{1+\alpha_0}
\right).
\]
\end{theorem}
\begin{remark}
$\Pi$ is a projective homeomorphism onto its image, with inverse $\Pi^{-1}\colon
\Pi(\Delta)\to\Delta$ given by
\[
\Pi^{-1}(x,y) = \left(
\frac{y}{1-y},\,\frac{1-x-2y}{1-y},\,\frac{x}{1-y}
\right).
\] 
\end{remark}
\begin{proof}	
We will prove in successive steps that
\[
\rho_8(t) = \rot_p(X,f_t,\Gamma) = \rot_p(X_t, f,\Gamma) =
\rot_p(Y_t, f, \Gamma) = \Pi(\rot_p(Z_t, F, \beta)) = \Pi(\DF(K(a(t)))),
\]
where $\beta\colon Z_t\to\Delta\subset\R^3$ is the observable which
takes the values $\be_0$, $\be_1$, and $\be_2$ in the intervals $C$,
$\one$, and $\two$ respectively. The first of these equalities, that
$\rho_8(t)=\rot_p(X,f_t,\Gamma)$, was established in
Section~\ref{sec:rot-sets-8}.

\medskip

\noindent\textbf{Step 1: $\rot_p(X, f_t, \Gamma) = \rot_p(X_t, f, \Gamma)$}

Suppose that $x\in X - X_t$ and that $\hGamma_{f_t}(x)$ exists. We
will find a point $y\in X_t$ with $\hGamma_{f_t}(y) =
\hGamma_{f_t}(x)$. This will establish that $\rot_p(X,f_t,\Gamma) =
\rot_p(X_t,f_t, \Gamma)$, and the result follows since $f=f_t$ on $X_t$.

Since $x\in X-X_t$ there is some $r\in\N$ for which $f_t^r(x)\in\oI$,
and hence $\hGamma_{f_t}(x) = \hGamma_{f_t}(\ell(t))$. If
$o(\ell(t),f_t) \cap\oI=\emptyset$ then we can take $y=\ell(t)\in
X_t$. Suppose, therefore, that there is some least $r\ge 1$
with $q_t:=f_t^r(\ell(t))\in\oI$. Then $f_t(q_t) = f_t(\ell(t))$, so
that $q_t$ is a period~$r$ point of~$f_t$. Since $f_t$ is locally
constant at $q_t$, we have $\operatorname{index}(q_t, f_t^r) =
+1$. Therefore $q_t$ can be continued to fixed points $q_s$ of $f_s^r$
for $s$ in a neighborhood of~$t$. Since $\hGamma_{f_s}(q_s) =
\widehat{\gamma^s}_{f_s}(q_s)$ has rational coordinates with denominator at
most~$r$, and varies continuously with $s$, we have
$\hGamma_{f_s}(q_s) = \hGamma_{f_t}(x)$ for all~$s$.

Let~$s$ be the smallest parameter for which the continuation $q_s$
exists. Then $q_s$ must be an endpoint of $I_s$, so that $q_s\in X_s
\subset X_t$. Taking $y=q_s\in X_t$ we have $\hGamma_{f_t}(y) =
\hGamma_{f_s}(q_s) = \hGamma_{f_t}(x)$ as required, since $f_t=f_s$ on $X_s$.

\medskip

\noindent\textbf{Step 2: $\rot_p(X_t, f, \Gamma) = \rot_p(Y_t, f, \Gamma)$}

Suppose that $x\in X_t$ and that $\hGamma_f(x)$ exists. We will find
a point $y\in Y_t$ with $\hGamma_f(y)=\hGamma_f(x)$, which will
establish the result.

Let $\uk\in\Sigma_8^+$ be an itinerary of~$x$, and let
$\uk'\in\Sigma_8^+$ be obtained by replacing every occurrence of $b$,
$c$, or $\twor$ in~$\uk$ with its orientation-preserving counterpart
$B$, $C$, or $\two$. Let $z\in X$ be a point with
itinerary~$\uk'$. Then we have
\begin{enumerate}[a)]
\item $o(z,f) \subset A\cup B\cup C\cup\one\cup\two$;
\item $\hGamma_f(z) = \hGamma_f(x)$; and
\item $z\in X_t$.
\end{enumerate}

Both a) and b) are obvious from the replacements which have been
carried out. For~c), observe that for each $r$, the points $f^r(x)$
and $f^r(z)$ lie on the same circle $S_1$ or $S_2$. Moreover
$f^r(z)\le f^r(x)$. For suppose $f^r(z)\not=f^r(x)$, and let $s \ge 0$
be least such that $k'_{r+s} \not= k_{r+s}$. Then $f^{r+i}(z)$ and
$f^{r+i}(x)$ pass through the same orientation-preserving intervals
for $0\le i < s$, and $f^{r+s}(z) < f^{r+s}(x)$.

It follows that $f^r(z)\not\in\oI$ for all~$r$; for otherwise we
would have $f^r(x)\in \two\cup\twor - \oI$, and hence
$f^{r+1}(z)>f^{r+1}(x)$. Therefore $z\in X_t$ as required.

\medskip

To complete the proof of step 2 we need to remove all occurrences of
the symbol~$A$ from~$\uk'$. We can assume that there are infinitely
many such, since otherwise we can take $y=f^r(z)$ for some $r$ large
enough that $\sigma^r(\uk')$ contains no symbol~$A$, and then $y\in
Y_t$ with $\hGamma_f(y) = \hGamma_f(z) = \hGamma_f(x)$. We can also
assume that $\uk'$ contains infinitely many symbols distinct from~$A$,
since otherwise $\hGamma_f(x) = (0,0)$ and we can choose~$y\in Y_t$ to
have itinerary $\overline{\one}$.

Write
\[\uk' = u_1\,C\,A^{L_1}\,B\,\,\,\, u_2\,C\,A^{L_2}\,B\,\,\,\, u_3\,C\,A^{L_3}\,B\,\,\,\,\ldots\]
where each $L_i\ge 1$ and the $u_i$ are (possibly empty) words
which do not contain the symbol~$A$. Since $\uk'$ doesn't contain the
symbols $b$, $c$, and $\twor$, there is a unique way to do this: each
maximal subword of the form $A^L$ must be preceded by $C$ and
followed by $B$, and the $u_i$ are the subwords which separate the
subwords $C\,A^L\,B$.

Let $\uk''$ be the sequence obtained by replacing each word $u_i\, C\,
A^{L_i}\, B$ with the word
\begin{itemize}
\item $u_i\, C\, B\, \one^{L_i}$ if $u_i$ does not contain the
  symbol~$\two$;
\item $w_1\,\one\, w_2\,\two\, C\,B\,\one^{L_i-1}$ if $u_i = w_1\,\two\,w_2$
  for some words $w_1$ and $w_2$ for which $w_2$ does not contain the symbol
  $\two$ (so that $u_i\, C\, A^{L_i}\, B = w_1\, \two\, w_2\, C\, A^{L_i}\, B$).
\end{itemize}

Then $\uk''\in\Sigma_8^+$ by choice of the replacement words, so that there is
a unique $y\in X$ with $f$-itinerary~$\uk''$. By the choice of the
replacement words we have $||\Gamma_f(y,r) - \Gamma_f(z,r)|| \le 1$
for all $r\in\N$, so that $\hGamma_f(y) = \hGamma_f(z) =
\hGamma_f(x)$. Since $y\in W$, it only remains to show that $y\in
X_t$; that is, that $f^r(y)\le \ell(t)$ whenever $f^r(y)\in\two$.

Suppose then that $k''_r = \two$. This symbol~$\two$ must be contained
in one of the replacement blocks
$w_1\,\one\,w_2\,\two\,C\,B\,\one^{L_i-1}$ of the second kind. If it
is contained in the word~$w_1$ then $f^r(y)\le f^r(z)\le\ell(t)$,
since the itineraries of $f^r(y)$ and $f^r(z)$ agree (and contain only
orientation-preserving symbols) up to the point where $f^{r+s}(y)\in
\one$ and $f^{r+s}(z)\in\two$. If it is not contained in the
word~$w_1$ then, since $w_2$ does not contain the symbol~$\two$, we
have $f^{r+1}(y)\in C$. Therefore $f^r(y)\le\ell(0)\le\ell(t)$, since
$f(\ell(0))\in\one$.

\medskip

\noindent\textbf{Step 3: $\rot_p(Y_t,f,\Gamma) = \Pi(\rot_p(Z_t, F, \beta))$}

We use Lemma~\ref{lem:induce}
applied to the return map $F\colon Z_t\to Z_t$ induced by $f\colon
Y_t\to Y_t$. The return time $N\colon Z_t\to\R$ is given by $N=2$ on
$C$ and $N=1$ on $\one\cup\two$; and the observable $\Phi\colon
Z_t\to\R^2$ corresponding to $\Gamma\colon Y_t\to\R^2$ takes values
$(0,1)$ on~$C$, $(0,0)$ on $\one$, and $(1,0)$ on $\two$. 

Now $\Phi = L\circ\beta$, where $L\colon\R^3\to\R^2$ is given by
$L(x,y,z) = (z,x)$; and $N = M\circ\beta$, where $M\colon\R^3\to\R$ is
given by $M(x,y,z) = 2x+y+z$. Now if $\bv\in\Delta$ then $\Pi(\bv) =
L(\bv)/M(\bv)$ (since $2v_0+v_1+v_2 = 1+v_0$). Since $\Pi$ is
injective, it follows from Lemma~\ref{lem:induce} that
$\rot_p(Y_t,f,\Gamma) = \Pi(\rot_p(Z_t, F, \beta))$ as required.

\medskip

\noindent\textbf{Step 4: $\Pi(\rot_p(Z_t, F, \beta)) = \Pi(\DF(K(a(t))))$}

This is immediate from Lemma~\ref{lem:conjtobeta}d) and the definition
$\DF(K(a(t))) = \rot_p(B(K(a(t))), \sigma, \kappa)$, since the
observables $\beta$ on $Z_t$ and $\kappa$ on $B(K(a(t)))$ correspond
under the conjugacy $h_t$.
\end{proof}

We now apply Theorem~\ref{thm:rotequal} in conjunction with
Theorem~\ref{thm:digit-freq} to describe the rotation sets
$\rho_8(t)$. Write $\cR = \{\rho_8(t)\,:\,t\in[0,1]\}\subset
\cH(\R^2)$, ordered by inclusion. We say that a vector
$\bv=(v_1,v_2)\in\R^2$ is {\em planar totally irrational} if $v_1$,
$v_2$, and $1$ are rationally independent. This condition is
equivalent to minimality of the translation of the torus whose lift is
$z\mapsto z+\bv$.

\begin{theorem} \mbox{}
\label{thm:eightrot}
\begin{enumerate}[a)]
\item Let~$t\in[0,1]$. Then $\rho_8(t) = \rot_p(X, \gamma^t,f_t) =
\rot_{\sMZ}(X,\gamma^t,f_t) = \rot_m(X, \gamma^t,f_t)$ is compact and
convex.
\item The map $\rho_8\colon [0,1]\to \cR$ is continuous and
  non-decreasing. 
\item The {\em bifurcation set}~$\cB\subset[0,1]$ of parameters~$t$ at which
  $\rho_8$ is not locally constant is a measure zero Cantor set. There
  is a partition $\cB=B_1\sqcup B_2\sqcup B_3$ with the following
  properties:
\begin{enumerate}[i)]
\item The set $B_1$ consists of the endpoints of the complementary
  gaps of~$\cB$. On each such gap, $\rho_8(t)$ is a constant polygon
  with rational vertices.
\item If $t\in B_2$ then $\Ex(\rho_8(t))$ consists of infinitely many
  rational polyhedral vertices together with a single irrational limit
  extreme point.
\item If $t\in B_3$ then $\Ex(\rho_8(t))$ consists of infinitely many
  rational polyhedral vertices together with two irrational limit
  extreme points, which are endpoints of a line segment in
  $\Bd(\rho_8(t))$.
\item Each of $B_2$ and $B_3$ is uncountable, and if $t\in B_2\cup
  B_3$ then $\rho_8(t)\not=\rho_8(t')$ for all $t'\not=t$.
\end{enumerate}
\item $\cR$ is order-preserving homeomorphic to a compact interval,
  and each $\rho_8(B_i)$ is dense in~$\cR$.
\item There is a dense $G_\delta$ subset of $\cR$, contained in
  $\rho_8(B_2)$, consisting of rotation sets whose limit extreme point
  is smooth and planar totally irrational.
\item The map $\Ex\circ\,\rho_8\colon [0,1]\to\cH(\R^2)$ is
  discontinuous at each point of $B_3$ and at the right-hand endpoints
  of the complementary gaps of~$\cB$, and is continuous elsewhere.
\end{enumerate}

\end{theorem}

\begin{proof}
Let $t\in[0,1]$. $\DF(K(a(t)))$ is compact and convex for each~$t$ by
Theorem~\ref{thm:digit-freq}a). It follows from
Theorem~\ref{thm:rotequal} and Lemma~\ref{lem:conjtobeta}c) that
$\rho_8(t):=\rot_p(X,\gamma^t,f_t)$ is also compact and convex, since
$\Pi$ is a projective homeomorphism onto its image. Equality with
$\rot_{\sMZ}(X,\gamma^t,f_t)$ and $\rot_m(X,\gamma^t,f_t)$ follows from
Lemma~\ref{lem:MZ}b) and the continuity of~$\gamma^t$. This establishes
part~a).

If $t\in\cC$ then $\rho_8(t) = \Pi(\DF(K(t)))$ by
Theorem~\ref{thm:rotequal} and the definition of~$a$. Moreover, if
$J=(t_1,t_2)$ is a complementary component of~$\cC$, then $\rho_8(t) =
\rho_8(t_1)$ for all $t\in J$. Since $K\colon \cC\to
\{\uw\in\Max\,:\,\uw\ge 2\overline{1}\}$ is an order-preserving
homeomorphism by Lemma~\ref{lem:conjtobeta}c), it follows that
$\rho_8(t_2) = \rho_8(t_1)$ also. In particular, because
$K(t_1)\not=K(t_2)$, Theorem~\ref{thm:digit-freq}c)iv) and c)i) give
that $\rho_8(t)$ is a constant polygon with rational vertices
for $t\in[t_1,t_2]$. In particular, $K(t_1)$ and $K(t_2)$ are consecutive
maximal sequences contained in one of the intervals of
Theorem~\ref{thm:digit-freq}c)i). 

Thus $\rho_8$ is constant on the closure of each complementary interval
of~$\cC$; while $\rho_8|_{\cC} = \Pi\circ\DF\circ K$ is continuous and
non-decreasing by Lemma~\ref{lem:conjtobeta}c) and
Theorem~\ref{thm:digit-freq}b). This establishes part~b).

The bifurcation set~$\cB$ of $\rho_8$ is therefore contained in~$\cC$,
and in particular has measure zero by Lemma~\ref{lem:conjtobeta}b). $\cB$
is the preimage of the bifurcation set of $\DF$, and can therefore be
partitioned as $\cB=B_1\sqcup B_2\sqcup B_3$, where (using the
notation of Theorem~\ref{thm:digit-freq}) $B_2 = K^{-1}(M_2)$,
$B_3 = K^{-1}(M_3)$, and
\[
B_1 = K^{-1}\left(\bigcup_{\uw\in M_1}\{u(\uw), v(\uw)\}\right).
\]
It is a Cantor set by Theorem~\ref{thm:digit-freq}c)v) and
Lemma~\ref{lem:conjtobeta}c). The remaining statements of the theorem
are just translations of the corresponding statements of
Theorem~\ref{thm:digit-freq}, using the fact that $\Pi$ is a projective
homeomorphism onto its image and the observation that an element $\bal$ of
$\Delta$ is totally irrational if and only $\Pi(\bal)$ is planar
totally irrational.
\end{proof}

\section{Rotation sets of a family of torus homeomorphisms}
\label{sec:torusrot}
In this section we will construct a continuously varying family
$\{\Phi_t\}$ of self-homeomorphisms of the torus whose rotation sets
$\rho(t)$ satisfy $\rho(t)=\rho_8(t)$ for all $t\in[0,1]$. To do this,
we use Theorem~3.1 of~\cite{param-fam} to ``unwrap'' the
family~$f_t$. This theorem is a generalization of a result of Barge
and Martin~\cite{bargemartin} to parameterized families. It states
(using definitions given below) that if~$\{f_t\}$ is a continuously
varying family of continuous self-maps of a {\em boundary retract}~$X$
of a manifold~$M$, satisfying a certain topological condition ({\em
  unwrapping}), then there is a continuously varying family
$\{\varphi_t\}$ of self-homeomorphisms of~$M$ such that $f_t$ and
$\varphi_t$ share their essential dynamical properties for
each~$t$. 

In Section~\ref{sec:unwrap} we state a version of the theorem which is
customized to the requirements of this paper. The theorem will then be
applied in Section~\ref{sec:torus-family} to construct the family of
torus homeomorphisms~$\{\Phi_t\}$ and show that
$\rho(\Phi_t)=\rho_8(t)$ for all~$t$. All parameterized families of
maps in this section will be assumed to have parameter~$t$ varying
over~$[0,1]$. 

\subsection{Unwrapping parameterized families}
\label{sec:unwrap}

\begin{defn}[Boundary retraction]
\label{defns:unwrap}
Let~$M$ be a compact manifold with non-empty boundary $\partial M$
and~$X$ be a compact subset of~$M$. A continuous map $\Psi\colon
\partial M\times[0,1]\to M$ is said to be a {\em boundary retraction
  of~$M$ onto~$X$} if it satisfies the following properties:
\begin{enumerate}[(1)]
\item $\Psi$ restricted to $\partial M \times[0,1)$ is a homeomorphism
  onto $M - X$,
\item $\Psi(\eta,0) = \eta$ for all $\eta\in\partial M$, and
\item $\Psi(\partial M \times\{1\}) = X$.
\end{enumerate}
\end{defn}

Therefore $\Psi$ decomposes $M$ into a continuously varying
family of arcs $\{\gamma_\eta\}_{\eta\in\partial M}$ defined by
$\gamma_\eta(s) = \Psi(\eta,s)$, whose images are mutually disjoint
except perhaps at their final points, which cover~$X$. In particular,
every point of $M-X$ can be written uniquely as $\Psi(\eta,s)$ with
$\eta\in\partial M$ and $s\in[0,1)$.

\begin{defns}[Unwrapping of a family, Associated family of
    near-homeomorphisms]

Let $\Psi\colon\partial M\times[0,1]\to M$ be a boundary retraction of
$M$ onto $X$, and $R\colon M\to X$ be the retraction defined by
$R(\Psi(\eta,s)) = \Psi(\eta,1)$. An {\em unwrapping} of a
continuously varying family $\{f_t\}$ of continuous maps
\mbox{$f_t\colon X\to X$} is a continuously varying family
$\{\barf_t\}$ of self-homeomorphisms of~$M$ with the property that,
for each~$t$,
\begin{enumerate}[(1)]
\item $R\circ \barf_t|_X = f_t$, and
\item $\barf_t$ is the identity on $\partial M$.
\end{enumerate}

Suppose that $\{\barf_t\}$ is an unwrapping of
$\{f_t\}$. Let $\lambda\colon[0,1]\to[0,1]$ be given by
$\lambda(s)=2s$ for $s\in[0,1/2]$ and $\lambda(s) = 1$ for $s\in[1/2,1]$,
and define $\Upsilon\colon M\to M$ by $\Upsilon(\Psi(\eta,s)) =
\Psi(\eta, \lambda(s))$, which is well defined since $\lambda(1)=1$. Write
$N(E) = \Psi(\partial M\times[1/2,1])$, a compact neighborhood of~$X$ which is
homeomorphic to $M$ by the homeomorphism $S\colon M\to N(E)$ defined
by $S(\Psi(\eta,s)) = \Psi(\eta, (s+1)/2)$, and satisfies $\Upsilon(N(E))=X$. Let
 $\{\barF_t\}$ be the family of self-homeomorphisms of~$M$
which is defined by $\barF_t = S\circ\barf_t\circ S^{-1}$ in $N(E)$, and $\barF_t
= \id$ in $M-N(E)$.

The {\em family of near-homeomorphisms $\{H_t\}$
  associated to the unwrapping $\{\barf_t\}$} is defined
by \[H_t = \Upsilon\circ\barF_t\colon M\to M.\]
\end{defns}

\begin{remarks}\mbox{}
\label{rmk:H_t}
\begin{enumerate}[a)]
\item $H_t|_X = R\circ \barf_t|_X = f_t$.
\item $H_t|_{\partial  M}$ is the identity.
\item If~$C$ is a compact subset of~$M$ disjoint from $\partial M$
  then there is some~$N\ge 0$ with $H_t^N(C)\subset X$ for all $t$,
  since $H_t(\Psi(\eta,s)) = \Psi(\eta,2s)$ if $s\le 1/2$ and
  $H_t(\Psi(\eta,s))\in X$ if $s\ge 1/2$.
\end{enumerate}
\end{remarks}

\begin{theorem}[\cite{param-fam}]
\label{thm:param-fam}
Let $M$ be a compact manifold with boundary $\partial M$, $\Psi$ be a
boundary retraction of $M$ onto a subset~$X$, and $\{f_t\}$ be a
continuously varying family of continuous surjections $X\to
X$. Suppose that an unwrapping $\{\barf_t\}$ of $\{f_t\}$ exists, and
let $\{H_t\}$ be the associated family of near-homeomorphisms.

  Then for every~$\veps>0$ there is a continuously varying family
  $\{\varphi_t\}$ of self-homeomorphisms of~$M$, and a
  Hausdorff-continuously varying family $\{\Lambda_t\}$ of compact
  $\varphi_t$-invariant subsets of~$M$ having the following properties
  for each $t\in[0,1]$.
\begin{enumerate}[a)]
\item There is a continuous map $g_t\colon M\to M$ within $C^0$
  distance $\veps$ of the identity, with $g_t(\Lambda_t) = X$, such
  that $H_t\circ g_t = g_t\circ\varphi_t$. In particular, $f_t\circ
  g_t|_{\Lambda_t} = g_t\circ\varphi_t|_{\Lambda_t}$.
\item $\varphi_t|_{\Lambda_t}$ is topologically conjugate to the
  natural extension $\hf_t$ acting on the inverse limit space
  $\varprojlim(X, f_t)$.
\item $\varphi_t$ is the identity on $\partial M$.
\item The non-wandering set $\Omega(\varphi_t)$ of $\varphi_t$ is contained in
  $\Lambda_t\cup\partial M$.
\end{enumerate}

\end{theorem}

Thus $\varphi_t|\Lambda_t$ is semi-conjugate to $f_t$, and $\Lambda_t$
contains all of the non-trivial recurrent dynamics of $\varphi_t$.

\medskip

That $H_t\circ g_t = g_t\circ \varphi_t$ is not contained
in theorem~3.1 of~\cite{param-fam}, but is explicitly
stated in its proof. Statement~d) of Theorem~\ref{thm:param-fam} is
slightly stronger than the corresponding statement
in~\cite{param-fam}, and we now sketch its proof. 

For each~$t$, let $M^t_\infty = \varprojlim(M, H_t)\subset M^\N$ be
the inverse limit of $H_t\colon M\to M$, and $\hatH_t\colon
M^t_\infty \to M^t_\infty$ be the natural extension of
$H_t$. Corollary~2.3 of~\cite{param-fam} provides a family of
homeomorphisms $h_t\colon M^t_\infty\to M$. In the proof of
theorem~3.1 of~\cite{param-fam}, the homeomorphisms $\varphi_t$ are
defined by $\varphi_t = h_t\circ \hatH_t\circ h_t^{-1}$, and the subsets
$\Lambda_t$ are given by $\Lambda_t = h_t(K_t)$, where
\[K_t = \{\ux\in M^t_\infty\,:\, x_k\in X \text{ for all }k\ge 0\}.\]
It therefore suffices to show that, for all $t$, $\Omega(\hatH_t)$ is
contained in the union of $K_t$ and
\[\partial M^t_\infty = \{\ux\in M^t_\infty\,:\,x_0\in\partial M\} =
\{\ux\in M^t_\infty\,:\,x_k\in\partial M \text{ for all }k\ge 0\}.\]
Now if $\ux\in M^t_\infty -  (K_t\cup\partial M^t_\infty)$ then
there is some~$k$ with $x_k\in M - (X \cup \partial M)$. Let~$C$
be a compact neighborhood of $x_k$ in~$M$ which is disjoint from
$X\cup \partial M$, and define 
\[U = \{\uy\in M^t_\infty\,:\,y_k\in C\},\]
a neighborhood of $\ux$ in $M^t_\infty$. Let~$N$ be large enough that
$H_t^N(C) \subset X$ (see Remark~\ref{rmk:H_t}c)): then
$\hatH_t^r(\uy)_k \in X$ for all $\uy\in U$ and $r\ge N$, so that
$\hatH_t^r(U)\cap U = \emptyset$ 
for all $r\ge N$ as required.

\subsection{The family of torus homeomorphisms}
\label{sec:torus-family}

Let $\pi\colon\R^2\to\T^2$ be the universal cover of the torus $\T^2 =
\R^2/\Z^2$, and let $M\subset \T^2$ be the torus with a hole obtained
by excising an open square~$S$ of side length $1/2$ centred in the
fundamental domain of the torus: that is,
\[
M = \T^2  -  S, \,\,\text{ where }\,\, S = \pi\left(\left\{
(x,y)\in\R^2\,:\, x\bmod 1\in(1/4,3/4) 
     \text{ and } y\bmod 1\in(1/4,3/4)
\right\}\right).
\]
We regard~$X = S_1\vee S_2$ as the subset of~$M$ given by $S_1 =
\pi([0,1]\times\{0\})$ and $S_2 = \pi(\{0\}\times[0,1])$. For each
$\eta\in\partial M$, let $\gamma_\eta\colon[0,1]\to M$ be the arc
in~$M$ whose image is a segment of the straight line passing through
the centre of the fundamental domain and $\eta$, parameterized
proportionally to arc length, so that $\gamma_\eta(0)=\eta$ and
$\gamma_\eta(1)\in X$ (see the dotted lines on
Figure~\ref{fig:unwrapping}). These arcs define a boundary retraction
$\Psi\colon \partial M\times[0,1] \to M$ of~$M$ onto~$X$, each point
of~$X$ being the endpoint of two of the arcs, with the exception of
the vertex~$v$ which is an endpoint of four arcs. The associated
retraction $R\colon M\to X$ is defined by $R(\Psi(\eta,s)) =
\Psi(\eta,1)$.

Let $\barf_1\colon M\to M$ be a homeomorphism unwrapping $f_1\colon
X\to X$ as depicted in Figure~\ref{fig:unwrapping}: the images of
$S_1$ and $S_2$ under $\barf_1$ are shown with solid and dashed lines
respectively, so that $R\circ\barf_1|_X = f_1$, and $\barf_1$ is then
extended arbitrarily to a homeomorphism $M\to M$ which is the identity
on $\partial M$. (Note that $\barf_1$ is injective on~$X$, since
$f_1(p)=p$, where $p$ is the common endpoint of the edges $\two$ and $\twor$
of $S_1$ as depicted in Figure~\ref{fig:graph-map}.) Postcomposing $\barf_1$
with a suitable isotopy supported in the disk~$D$ of
Figure~\ref{fig:unwrapping} yields an unwrapping $\{\barf_t\}$ of the
family~$\{f_t\}$.

\begin{figure}[htbp]
\includegraphics[width=0.45\textwidth]{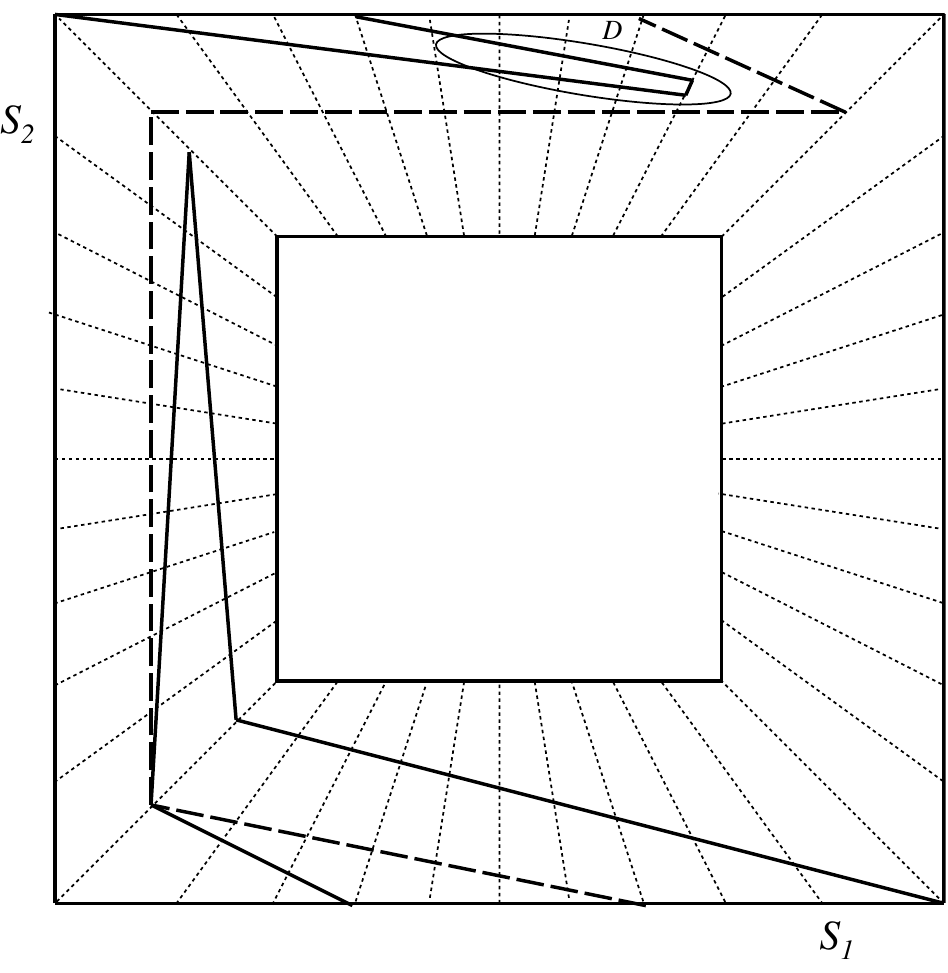}
\caption{Unwrapping the family $\{f_t\}$ in $\T^2 - S$}
\label{fig:unwrapping}
\end{figure}

Let $\{H_t\}$ be the family of near-homeomorphisms associated with the
unwrapping~$\{\barf_t\}$. Each~$H_t$ is homotopic to the identity,
since $H_t\vert_X = f_t$. Let $\veps<1/10$ be small enough that $d(H_t(x),
H_t(y))<1/10$ for all $t$ and all $x,y\in M$ with $d(x,y)<\veps$. 

Applying Theorem~\ref{thm:param-fam} with this value of $\veps$ yields
\begin{itemize}
\item A continuously varying family $\{\varphi_t\}$ of homeomorphisms
  $M\to M$, each the identity on $\partial M$;
\item A continuously varying family $\{\Lambda_t\}$ of compact
  $\varphi_t$-invariant subsets of~$M$ with the property that the
  non-wandering set $\Omega(\varphi_t)$ of $\varphi_t$ is contained in
  $\Lambda_t\cup\partial M$ for each~$t$; and
\item A continuous map $g_t\colon M\to M$ for each~$t$, within
  $C^0$-distance $\veps$ of the identity, satisfying
  \mbox{$g_t(\Lambda_t) = X$} and $H_t\circ g_t = g_t\circ\varphi_t$.  In
  particular, $f_t\circ g_t|_{\Lambda_t} = g_t\circ \varphi_t|_{\Lambda_t}$.
\end{itemize}

By choice of $\veps$ and the relationship $H_t\circ g_t =
g_t\circ\varphi_t$, the homeomorphism $\varphi_t$ is within
$C^0$ distance $1/5$ of the near-homeomorphism $H_t$, and is therefore
isotopic to the identity.

\begin{defns}[The family~$\{\Phi_t\}$, the displacement functions
    $\delta^t$, and the rotation sets~$\rho(t)$]
\mbox{}\\For each~$t\in[0,1]$, let
\begin{itemize}
\item $\Phi_t\colon\T^2\to\T^2$ be the homeomorphism obtained by
  extending $\varphi_t\colon M\to M$ as the identity across the
  excised square~$S$;
\item $\tPhi_t\colon\R^2\to\R^2$ be the lift of $\Phi_t$ which fixes
  the points of $\pi^{-1}(\overline{S})$;
\item $\delta^t\colon\T^2\to\R^2$ be the function defined by
  $\delta^t(x) = \tPhi_t(\tx) - \tx$, where $\tx$ is an arbitrary lift
  of~$x$; and
\item $\rho(t) = \rot_p(\T^2, \Phi_t, \delta^t)$ be the pointwise
  rotation set of $\Phi_t$ with respect to the lift $\tPhi_t$.
\end{itemize}
\end{defns}

\begin{theorem}
\label{thm:torusrot}
$\rho(t) =\rho_8(t)$ for each~$t\in[0,1]$. 

In particular, $\rho(t) = \rot_p(\T^2, \Phi_t, \delta^t) =
\rot_{\sMZ}(\T^2, \Phi_t, \delta^t) = \rot_m(\T^2, \Phi_t, \delta^t)$,
and all of the statements of Theorem~\ref{thm:eightrot} hold when
$\rho_8(t)$ is replaced with $\rho(t)$.
\end{theorem}

\begin{proof}
Since $\varphi_t$ is the identity on~$\partial M$, it follows that
$g_t(\partial M) \subset \partial M$. For if $x\in\partial M$ then
$H_t(g_t(x)) = g_t(\varphi_t(x)) = g_t(x)$, so that $g_t(x)$ is a
fixed point of~$H_t$. However $H_t$ has no fixed points outside of
$\partial M \cup X$ by construction; and $g_t(x)\not\in X$ since $d(x,
g_t(x))<\epsilon < 1/10$.

We can therefore extend $g_t$ to a continuous map
$g_t\colon\T^2\to\T^2$ by coning off its action on~$\partial M$. We
also extend $H_t$ as the identity across the excised square~$S$, to a
continuous map $H_t\colon\T^2\to\T^2$. Henceforth we use the symbols
$g_t$ and $H_t$ to refer to these continuous self-maps of the torus,
rather than to the original self-maps of~$M$. Since $g_t(\overline{S})
\subset \overline{S}$ and $\Phi_t|_S = H_t|_S = \id|_S$, we have
\[
H_t\circ g_t = g_t\circ\Phi_t \qquad\text{ for all }t\in [0,1].
\]
Note also that, since $\Omega(\varphi_t)\subset \Lambda_t\cup\partial
M$, we have $\Omega(\Phi_t) \subset \Lambda_t\cup \overline{S}$.

For each~$t$, let $\tH_t\colon\R^2\to\R^2$ be the lift of $H_t$ which
fixes the points of $\pi^{-1}(\overline{S})$. Then $\tH_t$ also fixes
the set $\pi^{-1}(v)$ of integer lattice points, since the vertex~$v$
of~$X$ is in the same $H_t$-Nielsen class as the points of
$\overline{S}$ by construction of~$H_t$. Let $\eta^t\colon\T^2\to\R^2$
be the displacement for $H_t$: that is, 
\[\eta^t(x) = \tH_t(\tx)-\tx,\]
where $\tx$ is an arbitrary lift of~$x$. Finally, let
$\tg_t\colon\R^2\to\R^2$ be the lift of~$g_t$ which is $\veps$-close
to $\id_{\R^2}$, so that
\[
\tH_t \circ \tg_t = \tg_t\circ\tPhi_t \qquad\text{ for all }t\in [0,1].
\]

We will prove in successive steps that
\[
\rot_p(X, f_t, \gamma^t) = \rot_p(\Lambda_t, \Phi_t, \eta^t\circ g_t)
= \rot_p(\Lambda_t, \Phi_t, \delta^t) = \rot_p(\T^2, \Phi_t, \delta^t),
\]
which will establish the equality of $\rho_8(t) =
\rot_p(X,f_t,\gamma^t)$ and $\rho(t) = \rot_p(\T^2, \Phi_t,
\delta^t)$. The equality of the pointwise, Misiurewicz-Ziemian, and
measure rotation sets then follows from Lemma~\ref{lem:MZ}b) and the
convexity of $\rho_8(t)$, completing the proof of the theorem.

\medskip

\noindent\textbf{Step 1: $\rot_p(X, f_t, \gamma^t) = \rot_p(\Lambda_t,
  \Phi_t, \eta^t\circ g_t)$}

Since $g_t(\Lambda_t) = X$ and $f_t\circ g_t|_{\Lambda_t} =
g_t\circ\Phi_t|_{\Lambda_t}$, it follows from~(\ref{eq:semiconj}) that
\[\rot_p(X, f_t, \gamma^t) = \rot_p(\Lambda_t, \Phi_t, \gamma^t\circ
g_t).\] However $\eta^t|_X = \gamma^t$, since $H_t|_X = f_t$ and the
lifts $\tH_t$ and $\tf_t$ both fix points above the vertex~$v$ of~$X$,
so that $\gamma^t\circ g_t|_{\Lambda_t} = \eta^t\circ
g_t|_{\Lambda_t}$.

\medskip

\noindent\textbf{Step 2: $\rot_p(\Lambda_t,
  \Phi_t, \eta^t\circ g_t) = \rot_p(\Lambda_t, \Phi_t, \delta^t)$}

For all $x\in\T^2$, $\tx\in\pi^{-1}(x)$, and $r\in\N$ we have
\begin{eqnarray*}
(\eta^t\circ g_t)_{\Phi_t}(x,r) &=& \sum_{i=0}^{r-1}
  \eta^t(g_t(\Phi_t^i(x)))\\
&=& \sum_{i=0}^{r-1} \left(\tH_t(\tg_t(\tPhi_t^i(\tx))) -
  \tg_t(\tPhi_t^i(\tx))\right)\\
&=& \sum_{i=0}^{r-1} \left(\tH_t^{i+1}(\tg_t(\tx)) -
  \tH_t^i(\tg_t(\tx))\right)\\
&=& \tH_t^r(\tg_t(\tx)) - \tg_t(\tx)\\
&=& \tg_t(\tPhi_t^r(\tx)) - \tg_t(\tx).
\end{eqnarray*}
Since $\delta^t_{\Phi_t}(x,r) = \tPhi_t^r(\tx) - \tx$ and $\tg_t$ is
$\veps$-close to the identity, $\eta^t\circ g_t$ and $\delta^t$ are
cohomologous with respect to $\Phi_t$, and the equality
follows.

\medskip

\noindent\textbf{Step 3: $\rot_p(\Lambda_t, \Phi_t, \delta^t) =
  \rot_p(\T^2, \Phi_t, \delta^t)$}

We have $\rot_p(\Lambda_t, \Phi_t, \delta^t) =
\rot_p(\Lambda_t\cup\overline{S}, \Phi_t, \delta^t)$, since
$\widehat{\delta^t}_{\Phi_t}(x) = (0,0)$ for all $x\in\overline{S}$,
and the rotation vector~$(0,0)$ is realized by any point $y\in\Lambda_t$
satisfying $g_t(y)=v$. Since $\rot_p(\Lambda_t, \Phi_t, \delta^t) =
\rho_8(t)$ is convex, and the recurrent set $R(\Phi_t)$ satisfies
$R(\Phi_t)\subset\Omega(\Phi_t)\subset \Lambda_t\cup\overline{S}$, the
equality follows from Lemma~\ref{lem:MZ}c).
\end{proof}

\begin{remark}
\label{rmk:higher-dimensions}
The results about digit frequency sets of symbolic $\beta$-shifts on
three symbols summarized in Theorem~\ref{thm:digit-freq} have analogs
for symbolic $\beta$-shifts defined over arbitrarily many
symbols~\cite{beta}. These can be used to compute the rotation sets of
families of self-maps of the wedge $X=S_1\vee S_2\vee\cdots\vee S_n$
of arbitrarily many circles defined analogously to the
family~$\{f_t\}$. These families can then be unwrapped to yield
families of homeomorphisms of $n$-dimensional tori whose rotation sets
agree with those of the self-maps of~$X$. The pointwise (or
Misiurewicz-Ziemian, or measure) rotation sets of these
higher-dimensional families then have properties analogous to those of
the family~$\{\Phi_t\}$ given in Theorem~\ref{thm:eightrot}. The only
differences are: the rotation sets are $n$-dimensional, and statements
about polygons should be replaced with statements about polytopes; in
the case $t\in B_3$, there can be between $2$ and $n$ irrational
extreme points; the authors have not proved a statement analogous to
the genericity of smooth limit extreme points; and neither have we
proved the discontinuity of the set of extreme points at parameters
in~$B_3$.

Here we sketch the changes which are required in the case~$n\ge 3$. We
subdivide the circle~$S_1$ into $2n+1$ oriented edges and the
other~$S_i$ into $3$ oriented edges:
\begin{eqnarray*}
S_1 &=&
C_2\,c_2\,C_3\,c_3\,\ldots\,C_{n}\,c_{n}\,D\,E\,e\quad\text{
  and}\\
S_i &=& A_i\,B_i\,b_i \qquad (2\le i\le n).
\end{eqnarray*}
The map $f\colon X\to X$ is defined by
\[
\begin{array}{lllll}
f(A_i) = S_i, \quad & f(B_i) = S_1, \quad & f(b_i) = S_1^{-1},&&
\\ f(C_i) = S_i, \quad & f(c_i) = S_i^{-1}, \quad & f(D) = S_1,
\quad & f(E) = S_1, \quad & f(e) = S_1^{-1},
\end{array}
\]
for $2\le i\le n$.  The family of maps $f_t\colon X\to X$ is then
defined by cutting off the tip of the transition $E\,e\mapsto
S_1\,S_1^{-1}$. Pointwise rotation sets $\rot_p(X, f_t, \gamma^t)$ are
defined by lifting to the abelian cover~$\tX$.

The analogs of Lemma~\ref{lem:conjtobeta} and
Theorem~\ref{thm:rotequal} are proved in exactly the same way. In
particular, when calculating rotation sets, it suffices to restrict to
those points whose orbits lie entirely in the edges $B_i$ ($2\le i\le
n$), $C_i$ ($2\le i\le n$), $D$ and $E$. Associating the symbol $i$ to
the word $C_{i+2}\,B_{i+2}$ for $0\le i\le n-2$, the symbol~$n-1$
to~$D$, and the symbol~$n$ to~$E$, reduces the calculation of these
rotation sets to that of digit frequency sets of symbolic
$\beta$-shifts on $n+1$ symbols.

Constructing a family $\{\Phi_t\}_{t\in[0,1]}$ of self-homeomorphisms
of~$\T^n$ with the same rotation sets proceeds exactly as in
Section~\ref{sec:torus-family}, using for the manifold~$M$ a tubular
neighborhood of~$X$. 

We thank the referee for pointing out that other useful
generalizations can be obtained using an embedded wedge of arbitrarily
many non-homotopic circles in~$\T^2$. On such a wedge, one could
define a family of maps which unwraps; or one could use Denjoy
examples as in~\cite{kwap2}.
\end{remark}

\section{Dynamical representatives of rotation vectors in symbolic
  $\beta$-shifts} 
\label{sec:representatives-beta}
In this section and the next we will study dynamical representatives
of elements of the rotation sets $\DF(\uw)$, $\rho_8(t)$, and
$\rho(t)$: that is, how elements of these sets are represented by invariant
sets and invariant measures of the underlying dynamical systems. The
simplest case, of digit frequency sets of symbolic $\beta$-shifts,
will be treated in this section, and the results applied in
Section~\ref{sec:representatives-toruseight} to the families~$\{f_t\}$
and $\{\Phi_t\}$ of maps of the figure eight space and the torus.

\subsection{Types of dynamical representatives of rotation vectors}
We start with some definitions and preliminary observations in the
general situation of Section~\ref{sec:rot-theory}. 

Let~$Z$ be a compact metric space, $g\colon Z\to Z$ be continuous, and
$\phi\colon Z\to\R^k$ be a continuous observable. Given an element
$\bv$ of the rotation set $\rot_p(Z,g,\phi)$, we first consider
invariant subsets in which every element~$z$ has rotation vector
$\hphi_g(z) = \bv$. We define three types of such subsets, with
increasingly strong properties: $\bv$-sets; $\bv$-minimal sets; and
\vgoober{s}.

\begin{defns}[$\bv$-set; $\bv$-minimal set; bounded deviation;
    \vgoober]\mbox{}
\begin{enumerate}[a)]
\item A {\em $\bv$-set} for $g$ with respect to $\phi$ is a non-empty
  $g$-invariant subset $Y$ of $Z$ with $\hphi_g(y)=\bv$ for all $y\in
  Y$. We say also that~$Y$ {\em represents} the rotation vector~$\bv$.
\item A {\em $\bv$-minimal set} for $g$ with respect to
  $\phi$ is a compact $g$-minimal $\bv$-set.
\item A point $z\in Z$ with $\hphi_g(z) = \bv$ is said to have {\em
  bounded deviation} (or {\em bounded mean motion}) if there is a
  constant~$M$ such that
\begin{equation}
\label{eq:bounded-dev}
|| \phi_g(z,r) - r\bv|| < M \qquad \text{ for all } r\ge 0.
\end{equation}
A $\bv$-set $Y$ for $g$ with respect to $\phi$ has {\em bounded
  deviation} if there is an~$M$ such that~(\ref{eq:bounded-dev}) holds
for all $z\in Y$.
\item A {\em \vgoober} for $g$ with respect
  to $\phi$ is a $\bv$-minimal set with bounded deviation.
\end{enumerate}
\end{defns}

\begin{remarks}\mbox{}
\label{rmk:v-sets}
\begin{enumerate}[a)]
\item Any compact $\bv$-set contains a minimal subset~$Y$, which
is therefore a $\bv$-minimal set.
\item A straightforward consequence of the continuity of~$g$
  and~$\phi$ is that if $z\in Z$ is any point for
  which~(\ref{eq:bounded-dev}) holds, its omega-limit set
  $\omega(z,g)$ is a $\bv$-set with bounded deviation
  (cf.~\cite{MZ2}). Hence, by~a), the existence of such a point~$z$
  implies the existence of a \vgoober.
\item The papers of J\"ager~\cite{J1,J2} explore the implications of
  bounded and unbounded deviation, showing in particular that if $\bv$
  is irrational then the dynamics on any \vgoober\ is semi-conjugate
  to rigid translation on a torus of some dimension. See
  Remark~\ref{rmk:compare}b).
\end{enumerate}
\end{remarks}

\noindent We next consider the representation of rotation vectors by
ergodic invariant measures.

\newpage 

\begin{defns}[Representation by ergodic invariant
    measures; directional and lost]\mbox{}\\ Let $\mu\in\cM_e(g)$.
\begin{enumerate}[a)]
\item The measure $\mu$ {\em represents} $\bv$ if $\bv =
  \rot(\mu,g,\phi)$. (That is, if $\bv = \int\phi\,d\mu$.) 
\item The measure $\mu$ is {\em ($\bv$-)directional} (for~$g$ with
  respect to~$\phi$) if its support $\supp(\mu)$ is a $\bv$-set, where
  $\bv=\rot(\mu,g,\phi)$; it is {\em lost} otherwise
  (\cite{gellmis}, cf.\ \cite{jenk, jenk2}).
\end{enumerate}
\end{defns}

\begin{remarks}\mbox{}
\begin{enumerate}[a)]
\item Using this terminology, Oxtoby's theorem~\cite{oxtoby} states that
  $\supp(\mu)$ is uniquely ergodic if and only if $\mu$ is directional
  for $g$ with respect to every continuous
  observable~$\phi$. For a measure to be directional can therefore be
  interpreted as an analog of unique ergodicity with respect to a
  single preferred observable.
\item Every $\bv$-minimal set~$Y$ is the support of a directional measure
  (namely any $\mu\in\cM_e(g|_Y)$); and, conversely, the support of
  any directional measure contains a $\bv$-minimal set by
  Remark~\ref{rmk:v-sets}a). 
\end{enumerate}

\end{remarks}

\subsection{Infimax minimal sets}
In the remainder of Section~\ref{sec:representatives-beta} we restrict
to the situation of Section~\ref{sec:dig-freq}, where
$Z=\Sigma^+=\{0,1,2\}^\N$ ordered lexicographically, the dynamics is
given by the shift map~$\sigma$, and the observable of interest is
$\kappa\colon\Sigma^+ \to\Delta\subset\R^3$ defined by $\kappa(\uw) =
\be_{w_0}$. We adopt the abbreviated notation
$\hkappa(\uw):=\hkappa_\sigma(\uw)$, $\DF_p(X) := \rot_p(X, \sigma,
\kappa)$, and $\DF_{\sMZ}(X) := \rot_{\sMZ}(X, \sigma, \kappa)$,
where~$X$ is a compact shift-invariant subset of~$\Sigma^+$. Recall
that if $\uw\in\Max$, we write $\DF(\uw)$ for $\DF_p(B(\uw))$, where
$B(\uw)$ is given by~(\ref{eq:betadef}).

The main tool used to analyse the digit frequency sets of symbolic
$\beta$-shifts is the {\em infimax sequences} introduced
in~\cite{lex}. We now summarize necessary results from that paper and
from~\cite{beta}. Define
\[
\Delta' = \{\bal=(\alpha_0,\alpha_1,\alpha_2)\in\Delta\,:\,\alpha_2>0\}.
\]

Given $\bal\in\Delta'$, 
write~$\cI(\bal)\in\Sigma^+$ for the infimum of the set of maximal
sequences $\uw\in\Max$ with $\hkappa(\uw)=\bal$. Note that $\cI(\bal)$
is necessarily maximal, since $\Max$ is closed in $\Sigma^+$, but need
not satisfy $\hkappa(\cI(\bal))=\bal$. The infimax
sequences~$\cI(\bal)$, which can be calculated using a
multidimensional continued fraction algorithm, have the additional
property (Lemma~19 of~\cite{lex}) that
\begin{equation}
\label{eq:infi-sup}
\sup(o(\uw, \sigma)) \ge \cI(\bal) \quad\text{ for all
  $\uw\in\Sigma^+$ with $\hkappa(\uw) = \bal$.} 
\end{equation}
  A consequence of this (Lemma~16 of~\cite{beta}) is that
\begin{equation}
\label{eq:df-char}
\bal\in\DF(\uw) \quad\iff\quad
\cI(\bal) \le \uw.
\end{equation}
Write
\[
\cJ = \Cl\left(\{\cI(\bal)\,:\,\bal\in\Delta'\}\right) \subset\Max,
\]
the closure in~$\Sigma^+$ of the set of infimax sequences. It is shown
in Lemmas~12 and~13 of~\cite{beta} that $\cJ$ is a Cantor set, and
that the only elements of~$\cJ$ which are not limits of both
strictly increasing and strictly decreasing sequences in~$\cJ$ are
\begin{itemize}
\item the elements added in the closure, which are not
  limits of any strictly increasing sequence in~$\cJ$, and
\item the elements $\cI(\bal)$ with $\bal\in\Q^3$, which are not
 limits of any strictly decreasing sequence in~$\cJ$.
\end{itemize}
 In particular, the supremum of any non-empty set of infimax sequences
 is itself an infimax sequence.
 
It follows from~(\ref{eq:df-char}) (see Lemma~19 of~\cite{beta}) that
for all $\uw\in\Max$,
\begin{equation}
\label{eq:infi}
\DF(\uw) = \DF(\cI(\bal)), \qquad\text{ where }\cI(\bal) =
\max\{\cI(\bal')\,:\,\bal'\in\Delta'\text{ with }\cI(\bal')\le \uw\}.
\end{equation}
It is then immediate from Theorems~\ref{thm:rotequal} and~\ref{thm:torusrot}
that, for all $t\in[0,1]$,  the rotation sets $\rho_8(t)=\rho(t)$ are of the
form $\Pi(\DF(\cI(\bal)))$ for some $\bal=\bal(t)\in\Delta'$.

\medskip

The three types of rotation sets which are described in
Theorems~\ref{thm:digit-freq} and~\ref{thm:eightrot} correspond to a
fundamental trichotomy for elements of~$\Delta'$:
\begin{description}
\item[Rational regular] When $\bal\in\Q^3$, the infimax sequence $\cI(\bal)$
  is periodic, and the digit frequency set $\DF(\cI(\bal))$ is a
  polygon with rational vertices, one of which is $\bal$.
\item[Irrational regular] {\em Regular} $\bal\not\in\Q^3$ are
  characterized by the property that $\cI(\bal)\not=\cI(\bal')$ for
  all $\bal'\not=\bal$. In this case $\cI(\bal)$ is aperiodic,
  $\hkappa(\cI(\bal)) = \bal$, and $\DF(\cI(\bal))$ has infinitely many
  rational polyhedral vertices limiting on the single irrational extreme
  point~$\bal$.
\item[Irrational exceptional] {\em Exceptional} $\bal\not\in\Q^3$ are
  characterized by the existence of a non-trivial {\em exceptional
    interval} $P_\bal\subset\Delta'$ with the property that
  $\cI(\bal)=\cI(\bal')$ if and only if $\bal'\in P_\bal$. In this
  case $\cI(\bal)$ is aperiodic, $\hkappa(\cI(\bal))$ does not exist,
  and $\DF(\cI(\bal))$ has infinitely many rational polyhedral vertices
  limiting on the two irrational endpoints of $P_\bal$.
\end{description}

The partition $\Max=M_1\sqcup M_2\sqcup M_3$ of
Theorem~\ref{thm:digit-freq} is connected to this trichotomy as
follows: $\uw\in M_2$ (respectively $\uw\in M_3$) if and only if $\uw = \cI(\bal)$ for some
irrational regular (respectively exceptional)~$\bal$; and $\uw\in M_1$
otherwise. The left hand endpoints of the intervals $I_{\uw}$ in $M_1$
are exactly the sequences $\cI(\bal)$ for $\bal\in\Q^3$; and the right
hand endpoints are exactly the non-infimax elements of~$\cJ$.

The following result, which is proved in Remark~23b) of~\cite{lex},
will play a central r\^ole.

\begin{lemma}
For every $\bal\in\Delta'$, the infimax sequence $\cI(\bal)$ is almost
periodic, and hence its orbit closure is a minimal set.
\end{lemma}

This motivates the following definition.

\begin{defn}[Infimax minimal sets]
For each $\bal\in\Delta'$, we define the {\em $\bal$-infimax minimal set}
\[
C_{\bal} = \Cl(\,o(\cI(\bal), \sigma)\,) \subset \Sigma^+.
\]
\end{defn}

By the definition of~$B(\uw)$, the maximality of $\cI(\bal)$,
and~(\ref{eq:df-char}), we have
\begin{equation}
\label{eq:C-alpha-detect}
C_{\bal} \subset B(\uw) \,\iff\, \cI(\bal) \le \uw \,\iff\, \bal\in\DF(\uw).
\end{equation}
Therefore the infimax minimal set~$C_\bal$ detects whether or not the
vector $\bal$ belongs to a given digit frequency set, and a natural
question is how well these sets represent~$\bal$. It is already clear
that $C_\bal$ is not an \algoober\ when $\bal$ is exceptional (since
it is not even an $\bal$-set), and
Example~\ref{ex:not-goober} below shows that this can be the case also
when $\bal$ is irrational regular.

The sets $C_\bal$ were studied in the context of attractors of
interval translation maps by Bruin and Troubetzkoy~\cite{BT}, and
their results about unique ergodicity are included in the following
lemma.

\begin{lemma}\mbox{}
\label{lem:BT}
\begin{enumerate}[a)]
\item If $\bal$ is regular, then $C_\bal$ is uniquely ergodic. In
  particular, $C_\bal$ is an $\bal$-minimal set, and the unique
  ergodic measure is $\bal$-directional.
\item If $\bal$ is exceptional, then $C_\bal$ has exactly two ergodic
  invariant measures $\mu_1$ and $\mu_2$, and $\rot(\mu_1, \sigma,
  \kappa)$ and $\rot(\mu_2, \sigma, \kappa)$ are the endpoints of the
  exceptional interval~$P_\bal$. In particular, $C_\bal$ is not an
  $\bal$-set, and $\mu_1$ and $\mu_2$ are both lost.
\item If $\bal$ is exceptional, then $\DF_{\sMZ}(C_\bal) = P_\bal$.
\end{enumerate}

\end{lemma}

\begin{proof}
Parts~a) and~b) are proved by Bruin and Troubetzkoy (Corollary~14 and
Lemma~17 of~\cite{BT}).

For part~c), we have $\Conv(\DF_{\sMZ}(C_\bal)) = P_\bal$ by part~b) and
Lemma~\ref{lem:MZ}a), so that $DF_{\sMZ}(C_\bal)\subset P_\bal$. For the
reverse inclusion, Theorem~57 of~\cite{beta}
states that there are subsequences of $(\kappa_\sigma(\cI(\bal),
r)/r)_{r\ge 1}$ converging to every point of $P_\bal$, so that $P_\bal
\subset \DF_{\sMZ}(C_\bal)$.
\end{proof}

\begin{example}
\label{ex:not-goober}
This example (cf.\ Lemma~52 of~\cite{beta}) shows that there exist
totally irrational regular values of~$\bal$ for which $C_\bal$ does
not have bounded deviation, and hence is not an \algoober.

Fix an integer $n\ge 1$, and let~$\Lambda_n\colon\Sigma^+\to\Sigma^+$
be the substitution defined by $0\mapsto 1$, $1\mapsto 20^{n+1}$, and
$2\mapsto 20^n$, with abelianization
\[
A_n = \left(
\begin{array}{ccc}
  0 & n+1 & n \\
  1 & 0 & 0\\
  0 & 1 & 1
\end{array}
\right),
\]
which is a Perron-Frobenius matrix since $A_n^3$ is strictly
positive. Let~$\bal$ be the positive Perron-Frobenius eigenvector
of~$A_n$ which satisfies $||\bal||_1 = 1$. Then we have
\begin{enumerate}[a)]
\item $\cI(\bal)$ is the fixed point $\lim_{r\to\infty}\Lambda_n^r(2)$
  of the substitution~$\Lambda_n$ (Theorem~22 of~\cite{lex});
\item $\bal$ is irrational regular (Theorem~27 of~\cite{lex});
\item $\bal$ is totally irrational (Lemma~52 of~\cite{beta}); and
\item $A_n$ has real eigenvalues $\lambda_1,\lambda_2,\lambda_3$ with
  $|\lambda_3|<1<|\lambda_2|<\lambda_1$ (proof of Lemma~52
  of~\cite{beta}).
\end{enumerate}
Write $r_i = ||A_n^i\be_2||_1$, so that $\kappa_\sigma(\cI(\bal),
r_i) = A_n^i \be_2$. Let $\bv_2$ be an eigenvector of~$A_n$
corresponding to the eigenvalue~$\lambda_2$. Then the magnitude of the
component of $\kappa_\sigma(\cI(\bal), r_i)$ in the direction
of~$\bv_2$ grows like $|\lambda_2|^i$, and in particular
\[
||\kappa_\sigma(\cI(\bal), r_i) - r_i\bal||\to\infty \quad\text{ as }i\to\infty,
\]
so that $\cI(\bal)$ does not have bounded deviation as required.  In
fact, we can be more explicit. By the Perron-Frobenius theorem, $r_i$
grows like $\lambda_1^i$. Writing $\nu =
\log(|\lambda_2|)/\log(\lambda_1)>0$ so that $|\lambda_2| =
\lambda_1^\nu$, we have
\[
||\kappa_\sigma(\cI(\bal), r_i) - r_i\bal|| > C_1|\lambda_2|^i > C_2 r_i^\nu
\]
for some positive constants $C_1$ and $C_2$.

That $C_\bal$ does not have bounded deviation also follows
from the much more general Theorem~1 of~\cite{adam}.
\end{example}

\subsection{Representation of elements of $\DF(\cI(\bal))$.}
Recall that every digit frequency set $\DF(\uw)$ is equal to
$\DF(\cI(\bal))$ for some~$\bal\in\Delta'$; and hence every rotation
set $\rho_8(t)=\rho(t)$ is equal to $\Pi(\DF(\cI(\bal)))$ for
some~$\bal\in\Delta'$. 

In this section we consider representations of elements of
$\DF(\cI(\bal))$ in $B(\cI(\bal))$. We prove two results:
Theorem~\ref{thm:alpha-rep}, concerning representations of~$\bal$ and
points on the exceptional interval $P_\bal$; and
Theorem~\ref{thm:v-rep}, concerning representations of other elements
of $\DF(\cI(\bal))$.

\medskip

We will need a preliminary lemma about infimaxes.
\begin{lemma}
\label{lem:supofcpt}
Let $\bal\in\Delta'$, and let~$X$ be a compact shift-invariant subset
of $\Sigma^+$ with $\bal\in\DF_{\sMZ}(X)$. Then $\max X\ge \cI(\bal)$.
\end{lemma}

\begin{proof}
Suppose for a contradiction that $\max X < \cI(\bal)$. Then $X \subset
B(\cI(\bal))$ by definition of the symbolic $\beta$-shift, and hence
\[
\bal \in \DF_{\sMZ}(X) \subset \Conv(\DF_{\sMZ}(X)) \subset
\Conv(\DF_{\sMZ}(B(\cI(\bal)))) = \DF(\cI(\bal))
\]
by definition of the convex set $\DF(\cI(\bal)) := \rot_p(B(\cI(\bal)), \sigma,
\kappa) = \rot_{\sMZ}(B(\cI(\bal)), \sigma, \kappa)$.

By Carath\'eodory's theorem, we can find three extreme points
$\bbeta_i$ ($1\le i\le 3$) of $\Conv(\DF_{\sMZ}(X))$ which
contain~$\bal$ in their convex hull. 
By Lemma~\ref{lem:MZ}a), there are elements
$\uw_i$ of~$X$ with $\hkappa(\uw_i) = \bbeta_i$ for each~$i$.

Since $\bbeta_i\in\Conv(\DF_{\sMZ}(X))\subset \DF(\cI(\bal))$, and
$\bal$ is either an extreme point of $\DF(\cI(\bal))$ (in the regular case), or
lies in the exceptional interval $P_\bal$ on the boundary of
$\DF(\cI(\bal))$, at least one of the points $\bbeta_i$ is either
equal to $\bal$, or lies in $P_\bal$. In either case, we have
$\cI(\bbeta_i) = \cI(\bal)$. Therefore, since $\uw_i\in X$ has
$\hkappa(\uw_i)=\bbeta_i$, (\ref{eq:infi-sup}) gives
\[
\max X \ge \sup(o(\uw_i, \sigma)) \ge \cI(\bbeta_i) = \cI(\bal),
\]
which is the required contradiction.
 \end{proof}

\begin{remark} 
Since $C_\bal$ is a minimal subset of $\Sigma^+$ with $\bal\in\DF_{\sMZ}(C_{\bal})$,
Lemma~\ref{lem:supofcpt} gives the characterisation
\[
\cI(\bal) = \min\{\max(C)\,:\, C\subset\Sigma^+ \text{ is a minimal
  set with } \bal\in\DF_{\sMZ}(C)\}.
\]
\end{remark}

\begin{theorem}
\label{thm:alpha-rep}
Let $\bal\in\Delta'$. Then any minimal subset~$C$ of $B(\cI(\bal))$
which satisfies
\[
\DF_{\sMZ}(C) \cap \DF_{\sMZ}(C_{\bal}) \not=\emptyset
\]
is equal to $C_{\bal}$. In particular,
\begin{enumerate}[a)]
\item if $\bal$ is regular, then $C_\bal$ is the unique $\bal$-minimal
  set in $B(\cI(\bal))$; and
\item If $\bal$ is exceptional, then there are no $\bal'$-minimal sets
  in $B(\cI(\bal))$ for any $\bal'\in P_\bal$.
\end{enumerate}
\end{theorem}

\begin{proof}
By Lemma~\ref{lem:BT}, $\DF_{\sMZ}(C_\bal)$ is equal either to
$\{\bal\}$ (in the regular case), or to $P_\bal$. In either case,
every element $\bal'$ of $\DF_{\sMZ}(C_\bal)$ has
$\cI(\bal')=\cI(\bal)$.

Therefore if~$C$ is any minimal subset of $B(\cI(\bal))$ for which
$\DF_{\sMZ}(C)$ intersects $\DF_{\sMZ}(C_\bal)$, then there is some
$\bal'\in\DF_{\sMZ}(C)$ with $\cI(\bal')=\cI(\bal)$. By
Lemma~\ref{lem:supofcpt}, it follows that $\max C \ge \cI(\bal)$.

On the other hand, $\max C \le \cI(\bal)$, since $C\subset
B(\cI(\bal))$. Therefore $\max C = \cI(\bal)$, i.e. $\cI(\bal)\in
C$. By minimality, it follows that $C = C_\bal$ as required.

The statements~a) and~b) are immediate consequences, using the facts
that $\DF_{\sMZ}(C_\bal)$ is either $\{\bal\}$ or $P_\bal$; and that
$C_\bal$ is not an $\bal$-set in the exceptional case.
\end{proof}

By Theorem~\ref{thm:alpha-rep}, and in view of
Example~\ref{ex:not-goober}, we cannot expect every irrational $\bal$,
or any element of an exceptional interval $P_\bal$, to be represented
by a \plaingoober\ in $B(\cI(\bal))$. The next theorem says that all
other $\bv$ can be represented by a \vgoober. Write $Q_\bal =
\emptyset$ if $\bal$ is rational regular; $Q_\bal=\{\bal\}$ if $\bal$
is irrational regular; and $Q_\bal = P_\bal$ if	 $\bal$ is irrational
exceptional.
\begin{theorem}
\label{thm:v-rep}
Let $\bal\in\Delta'$.  Then every $\bv\in\DF(\cI(\bal)) - Q_{\bal}$ is
represented by a \vgoober\ in~$B(\cI(\bal))$.
\end{theorem}

\begin{proof}
If $\bv\in\DF(\cI(\bal))$ is rational, then $C_\bv =
o(\cI(\bv),\sigma)$ is a periodic orbit (and hence a \vgoober) which
is contained in $B(\cI(\bal))$ by~(\ref{eq:C-alpha-detect}). 

For a rational element $\bu$ of $\Delta'$, let $W_\bu$ be the repeating
block of the periodic sequence $\cI(\bu)$. Lemma~7 of~\cite{beta}
states that if $(\bu_i)$ is a sequence of rational elements
of~$\Delta'$ with $\bu_i\in\DF(\cI(\bal))$ for each~$i$, then the
sequence
\[
\uw = W_{\bu_0}\,W_{\bu_1}\,W_{\bu_2}\,\ldots \,\in\Sigma^+
\]
is an element of $B(\cI(\bal))$. 

Therefore if $\bv\in\Intt(\DF(\cI(\bal)))$ then we can choose
rational elements~$\bu_0$, $\bu_1$, and $\bu_2$ of $\DF(\cI(\bal))$,
with the same denominator, which contain $\bv$ in their convex
hull. The concatenation scheme of the blocks $W_{\bu_0}$, $W_{\bu_1}$,
and $W_{\bu_2}$ used by Ziemian in the proof of Lemma~4.4 of~\cite{Z}
can then be used to construct a \vgoober\ in $B(\cI(\bal))$.

It therefore only remains to consider the case where
$\bv\in\Bd(\DF(\cI(\bal))) - Q_\bal$ is irrational. Then either $\bv$
is contained in an interval~$I$ in~$\Bd(\DF(\cI(\bal)))$ whose
endpoints are rational polyhedral vertices, or $\bal$ is irrational
regular and $\bv$ is contained in an interval~$I$
in~$\Bd(\DF(\cI(\bal)))$ for which one endpoint is a rational
polyhedral vertex, and the other endpoint is~$\bal$. However, in the
latter case, where the vertices of $\DF(\cI(\bal))$ limit on~$\bal$
from one side only, the interval~$I$ is contained in an interval with
rational endpoints (see Example~36c) of~\cite{beta}). Therefore, in
either case, we can find rational elements $\bu_0 = \bp_0/q$ and
$\bu_1 = \bp_1/q$ of $\DF(\cI(\bal))$, having the same
denominator~$q$, such that
\[
\bv = (1-\lambda)\,\frac{\bp_0}{q} \,+\, \lambda\,\frac{\bp_1}{q}
\]
for some irrational~$\lambda\in[0,1]$. These rational vectors
$\bp_i/q$ need not be in reduced form: let $k_0$ and $k_1$ be
the positive integers with the property that when $\bp_i/q$ is written
in reduced form it has denominator $q/k_i$ (so that the word
$W_{\bu_i}^{k_i}$ has length~$q$).

Let $\us\in\{0,1\}^\N$ be the Sturmian sequence~\cite{sturm} of
slope~$\lambda$, and let $\uw\in B(\cI(\bal))$ be the element of
$\Sigma^+$ obtained from $\us$ by the substitution $0\mapsto
W_{\bu_0}^{k_0}$, $1\mapsto W_{\bu_1}^{k_1}$. Since $\us$ is Sturmian,
we have $|a_r-r\lambda|<1$ for all $r\in\N$, where $a_r$ denotes the
number of $1$s in the first~$r$ symbols of~$\us$. Therefore, for
each~$r\in\N$,
\[
||\kappa_\sigma(\uw, qr) - qr\bv|| = ||a_r\bp_1 + (r-a_r)\bp_0 -
qr\bv|| = ||(a_r-r\lambda)(\bp_1-\bp_0)|| < ||\bp_1-\bp_0|| < q.
\]
If~$s$ is any natural number then, writing $s=qr+t$ with $0\le t<q$,
we have 
\[
||\kappa_\sigma(\uw,s) - s\bv || \le ||\kappa_\sigma(\uw,qr) - qr\bv||
+ ||\kappa_\sigma(\sigma^{qr}(\uw), t) - t\bv|| < 2q.
\] 
The existence of a \vgoober\ follows by Remark~\ref{rmk:v-sets}b). (In
fact, since $\uw$ is almost periodic by construction, its orbit
closure is equal to its omega-limit set, and hence is itself a
\vgoober.)
\end{proof}

\section{Dynamical representatives of rotation vectors in the figure
  8 and torus families}
\label{sec:representatives-toruseight}
\subsection{Representatives in the figure eight family}
Because the return map $R_t\colon Z_t\to Z_t$ of the restriction of
the figure eight map $f_t$ to the subset $Z_t$ of the figure eight
space~$X$ is topologically conjugate to a symbolic $\beta$-shift
(Lemma~\ref{lem:conjtobeta}d), translating results about the existence
of $\bv$-sets and directional measures to this context is
straightforward. On the other hand, uniqueness and non-existence
results need an additional result (Lemma~\ref{lem:extrememin} below)
to rule out invariant subsets which are not contained in~$Z_t$.

Recall that $\cC\subset[0,1]$ is the compact set of parameters~$t$ for
which $\ell(t)\in Z_t$; and that there is a ``kneading sequence'' map
$K\colon\cC\to\Sigma_+$, which is an order-preserving homeomorphism
onto the set $\{\uw\in\Max\,:\,\uw\ge 2\overline{1}\}$ of sufficiently
large maximal sequences. It is a consequence of Theorem~22
of~\cite{lex} that an infimax sequence $\cI(\bal)$ lies in this set
if and only if $\alpha_0<\alpha_2$, and we therefore restrict
attention to the subset
\[
\Delta'' = \{\bal\in\Delta\,:\,\alpha_0<\alpha_2\}
\]
of $\Delta'$. Given $\bal\in\Delta''$, let \[t(\bal) =
K^{-1}(\cI(\bal))\in\cC\] be the unique parameter~$t$ for which
$K(t)=\cI(\bal)$.

Recall that the function $a\colon[0,1]\to\cC$ is defined by $a(t) =
\max\{t'\in\cC\,:\,t'\le t\}$. We now define $b\colon[0,1]\to\cC$ by
\[
b(t) = 
\left\{
\begin{array}{ll}
\max\{t'\in\cC\,:\,t'\le t \text{ and }K(t')\text{ is
 an infimax sequence}\} \quad& \text{ if } t>0,\\
0 & \text{ if }t=0.
\end{array}
\right.
\]
The maximum exists since the set of infimax sequences which are less
than or equal to~$K(t)$ has a maximum (as in~(\ref{eq:infi})) and~$K$
is an order-preserving homeomorphism. It is immediate from the
definitions that $b(a(t))=a(b(t))= b(t)$ for all~$t$.

\begin{lemma}
\label{lem:alem} 
Let $t\in [0,1]$. Then $\rho_8(t) = \rho_8(b(t)) = \Pi(\DF(K(b(t))))$.
\end{lemma} 
\begin{proof}
Let $\cI(\bal)$ be the greatest infimax sequence which is not greater
than $K(a(t))$. Then we have $\cI(\bal) = K(b(a(t)))=K(b(t))$ by definition of
the function~$b$ and the injective monotonicity of~$K$.

It follows that
\[
\rho_8(t) = \rho_8(a(t)) = \Pi(\DF(K(a(t)))) = \Pi(\DF(\cI(\bal))) =
\Pi(\DF(K(b(t)))) = \rho_8(b(t)),
\]
as required. Here the first, second, and last equalities are given by
Theorem~\ref{thm:eightrot}; the third is by~(\ref{eq:infi}); and the
fourth is by the first paragraph of the proof.
\end{proof}

\begin{remark} Lemma~\ref{lem:alem} shows that the bifurcation
set $\cB$ of parameters at which $\rho_8$ is not locally constant (see
Theorem~\ref{thm:eightrot}) is given by $\cB = K^{-1}(\cJ)$.
\end{remark}

For each $\bal\in\Delta''$, define $D_\bal\subset X$ by
\[
D_\bal = \Cl\left(o\,(\ell(t(\bal)),\, f_{t(\bal)})\,\right).
\]
Since the return map $R_{t(\bal)}\colon Z_{t(\bal)}\to Z_{t(\bal)}$
and $\sigma\colon B(\cI(\bal))\to B(\cI(\bal))$ are conjugate by
Lemma~\ref{lem:conjtobeta}d) and the definition of~$t(\bal)$, $D_\bal$
is conjugate to a two step tower over the $\sigma$-minimal set~$C_\bal$,
and is therefore a minimal set for $f_{t(\bal)}$.

\begin{lemma}
\label{lem:extrememin}
Let $\bal\in\Delta''$, and $D$ be a minimal set for $f_{t(\bal)}$
with the property that $\Pi(\bal)\in\rot_{\sMZ}(D, f_{t(\bal)},
\Gamma)$. Then $D=D_\bal$.
\end{lemma}

\begin{proof}
Recall from Section~\ref{sec:inv-sub} that $X_t\subset X$ is defined
for each $t\in[0,1]$ by \[X_t=\{x\in X\,:\,o(x,f_t)\subset
X-(\ell(t),r(t))\,\}.\]

Suppose that $t<t(\bal)$. Then $b(t)<t(\bal)$, so that
$K(b(t))<\cI(\bal)$, and hence $\Pi(\DF(K(b(t))))$ does not contain
$\Pi(\bal)$ by~(\ref{eq:df-char}). It follows that $D$ cannot be
contained in $X_t$, since this would yield the contradiction
\[\Pi(\bal)\in\rot_{\sMZ}(D, f_{t(\bal)}, \Gamma) \subset \rot_{\sMZ}(X_t,
f, \Gamma) = \rho_8(t) = \Pi(\DF(K(b(t)))),\] 
using Lemma~\ref{lem:alem} and that $f_{t(\bal)}=f_t=f$ on $X_t$.

It follows then by the compactness of~$D$ that it contains a point $x\in
[\ell(t(\bal)), r(t(\bal))]$. Since $f_{t(\bal)}(x) =
f_{t(\bal)}(\ell(t(\bal))) \in D\cap D_\bal$, the minimal sets $D$
and $D_\bal$ are equal as required.
\end{proof} 

Since all of the rotation sets $\rho_8(t)$ are of the form
$\Pi(\DF(\cI(\bal)))$ for some $\bal\in\Delta''$, and $\Pi$ is
projective, they can be classified as either rational regular
(i.e.\ polygonal), irrational regular (i.e.\ having a single irrational
extreme point~$\bu_t=\Pi(\bal)$), or irrational exceptional (i.e.\ having two
irrational extreme points which are the endpoints of an exceptional
interval $P_t = \Pi(P_\bal)$ in the boundary of $\rho_8(t)$).

\begin{theorem}
\label{thm:eight2}
Let $t\in [0,1]$. 
\begin{enumerate}[a)]
\item If $\rho_8(t)$ is rational regular, then every $\bv\in\rho_8(t)$
  is represented by a \vgoober\ for~$f_t$. Thus, $\rho_{em}(t) =
  \rho_8(t)$.
\item If $\rho_8(t)$ is irrational regular then:
\begin{enumerate}[i)]
\item Every $\bv\in\rho_8(t)$ is represented by a $\bv$-minimal
  set for $f_t$. In particular, $\rho_{em}(t)=\rho_8(t)$.
\item Every $\bv\in\rho_8(t) - \{\bu_t\}$ is represented by a
  \vgoober\ for $f_t$.
\end{enumerate}
There exist irrational regular $\rho_8(t)$ for which
 $f_t$ has no \goober{$\bu_t$}.
\item If $\rho_8(t)$ is irrational exceptional then:
\begin{enumerate}[i)]
\item Every $\bv\in\rho_8(t) - P_t$ is represented by a \vgoober\ for~$f_t$.
\item There is a unique $f_t$-minimal set~$D$ such that
  $\rot_{\sMZ}(D,f_t,\Gamma)$ intersects $P_t$, namely $D=D_\bal$
  where $t=t(\bal)$. In particular, if $\bv\in P_t$ then there is no $\bv$-minimal
  set for~$f_t$.
\item Any ergodic invariant measure for~$f_t$ representing any vector
  $\bv\in P_t$ is lost. In particular, the two ergodic invariant
  measures on $D_\bal$, which represent the endpoints of~$P_t$, are
  lost.
\end{enumerate}

\end{enumerate}

\end{theorem}

\begin{proof}
Let $\bal\in\Delta''$, and write $t:=t(\bal)$, so that $K(t) =
\cI(\bal)$, and $\rho_8(t) = \Pi(\DF(\cI(\bal)))$ by
Lemma~\ref{lem:alem}.

Lemma~\ref{lem:conjtobeta}d) states that $h_t\colon Z_t\to
B(\cI(\bal))$ conjugates the return map $R_t\colon Z_t\to Z_t$ and the
symbolic $\beta$-shift $\sigma\colon B(\cI(\bal))\to
B(\cI(\bal))$. There is therefore a bijection $M_t$ from the set of
minimal sets of $\sigma\colon B(\cI(\bal))\to B(\cI(\bal))$ to the set
of minimal sets of $f_t\colon Y_t\to Y_t$, which sends the minimal set
containing~$\uw$ to the minimal set containing
$h_t^{-1}(\uw)$. Moreover, $C$ is a $\bv$-minimal set (respectively \vgoober) for $\sigma$ if
and only if $M_t(C)$ is a $\Pi(\bv)$-minimal set (respectively \goober{$\Pi(\bv)$}) for~$f_t$.

If $t\in[0,1]$ is not of the form $t(\bal)$ for any $\bal\in\Delta''$,
then $\rho_8(t)=\rho_8(b(t))$ is rational regular. Since $b(t)<t$, we
have that $f_t = f_{b(t)}$ on $Y_{b(t)}$, and hence any
$f_{b(t)}$-minimal set in $Y_{b(t)}$ is also an $f_t$-minimal set.

Parts a), b)ii), and c)i) therefore follow from
Theorem~\ref{thm:v-rep}, and part b)i) follows from
Theorem~\ref{thm:alpha-rep}a). 

For the final statement of part b), pick $\bal$ to be the normalized
Perron-Frobenius eigenvalue of $A_1$ as in
Example~\ref{ex:not-goober}. This $\bal$ does not lie in $\Delta''$,
but $\bbeta = (\alpha_1, \alpha_0, \alpha_1+\alpha_2)/(1+\alpha_1) \in
\Delta''$ has analogous properties (in the language of~\cite{lex},
$\bbeta$ has {\em itinerary} $0\overline{1}$). In particular, $\bbeta$
is totally irrational regular, and $\cI(\bbeta)$ is obtained by
applying $\Lambda_0$ to the fixed point of the
substitution~$\Lambda_1$; and, by the same argument as in
Example~\ref{ex:not-goober}, $C_\bbeta$ is not a
\goober{$\bbeta$}. Therefore $D_\bbeta = M_{t(\bbeta)}(C_\bbeta)$ is not a
\goober{$\bu_{t(\bbeta)}$} for~$f_{t(\bbeta)}$. Since~$D_\bbeta$ is
the unique $\bu_{t(\bbeta)}$-minimal set for~$f_{t(\bbeta)}$ by
Lemma~\ref{lem:extrememin}, the result follows.

Parts c)ii) and c)iii) follow from Lemma~\ref{lem:extrememin},
Theorem~\ref{thm:alpha-rep}b), and Lemma~\ref{lem:BT}b).
\end{proof}

\subsection{Representatives in the torus family}
Recall from Theorem~\ref{thm:param-fam} and the construction of
Section~\ref{sec:torus-family} that the homeomorphism
$\Phi_t\colon\T^2\to\T^2$ has non-wandering set contained in
$\Lambda_t \cup \overline{S}$; that~$\overline{S}$ consists of fixed
points of $\Phi_t$; and that $\Phi_t|_{\Lambda_t}$ is topologically
conjugate to the natural extension $\hat{f}_t\colon
\varprojlim(X,f_t)\to\varprojlim(X,f_t)$ of the figure eight
map~$f_t$. 

In order to extend the results of Theorem~\ref{thm:eight2} to the
torus family, it is therefore only necessary to understand the
relationship between minimal sets and ergodic invariant measures of a
map and its natural extension. Given a continuous map $g\colon Z\to Z$
of a compact metric space, we write
$\hg\colon\varprojlim(Z,g)\to\varprojlim(Z,g)$ for its natural
extension, and $\pi_0\colon \varprojlim(Z,g)\to Z$ for the projection
$(z_0,z_1,\ldots)\mapsto z_0$. We also denote by $\Min(g)$ the
collection of $g$-invariant minimal subsets of~$Z$.

The following result is folklore: see page 28 of~\cite{furst},
Theorem~3.2 of~\cite{para}, \cite{brown} and~\cite{kenn}.

\begin{lemma}
\label{lem:invlim}
Let $g\colon Z\to Z$ be a continuous map of a compact metric
space. Then the maps $C \mapsto \pi_0(C)$ and $\mu \mapsto
(\pi_0)_\ast(\mu)$ are bijections $\Min(\hg)\to\Min(g)$ and
\mbox{$\cM_e(\hg)\to\cM_e(g)$} respectively.
\end{lemma}
We then immediately have:
\begin{theorem}
\label{thm:torus2}
All of the statements of Theorem~\ref{thm:eight2} hold when $f_t$ is
replaced by $\Phi_t$ and $\rho_8(t)$ is replaced by $\rho(t)$.
\end{theorem}

\begin{remarks}
\label{rmk:compare}
\begin{enumerate}[a)]
\item Theorem~\ref{thm:torus2} was proved directly as a translation of
  Theorems~\ref{thm:alpha-rep} and~\ref{thm:v-rep} about rotation sets
  of symbolic $\beta$-shifts. Much is already known, however, about
  representatives for rotation vectors in the interior of the rotation
  set of a torus homeomorphism $\Phi\colon\T^2\to\T^2$ isotopic to the
  identity. Misiurewicz and Ziemian show that every
  $\bv\in\Intt(\rho_{\sMZ}(\Phi))$ is represented by a
  \vgoober\ (proof of Theorem~A(a) of~\cite{MZ2}). When~$\bv$ is
  rational, these \vgoober{s} can be chosen to be periodic orbits, by
  a theorem of Franks~\cite{franks}; and to have the topological type
  of a periodic orbit of a rigid rotation of the torus by the
  vector~$\bv$, using a result of Parwani~\cite{parwani}.
\item The significance of the existence of \vgoober{s} is illustrated
  by J\"ager's result~\cite{J1}, that if $\bv$ is irrational
  and~$\Phi$ has a \vgoober~$Z$, then there is a semi-conjugacy,
  homtopic to the inclusion, onto a minimal set of rigid rotation
  by~$\bv$. Therefore if~$\bv$ is totally irrational then $\Phi\colon
  Z\to Z$ is semi-conjugate to a minimal rigid rotation of~$\T^2$,
  while if~$\bv$ is partially irrational then $\Phi\colon Z\to Z$ is
  semi-conjugate to a minimal rigid rotation of a circle.
\item The relationship between Theorem~\ref{thm:torus2} and recent
  results of Zanata~\cite{zanata} and Le Calvez \& Tal~\cite{tal} is
  also important to note.  If $\Phi\colon\T^2\to\T^2$ is a
  homeomorphism and $\mu\in\cM_e(\Phi)$ represents
  an extreme point~$\bv$ of $\rho_{\sMZ}(\Phi)$, then
\begin{enumerate}[i)]
\item if $\rho_{\sMZ}(\Phi)$ has multiple supporting lines at~$\bv$,
  then $\mu$ is directional and its support has bounded deviation; and
\item if $\rho_{\sMZ}(\Phi)$ has a unique supporting line at $\bv$ which
  {\em does not} intersect $\rho_{\sMZ}(\Phi)$ in a non-trivial segment,
  then $\mu$ is directional.
\end{enumerate}
These statements were proved by Zanata in the case where $\Phi$ is a
$C^{1+\epsilon}$ diffeomorphism, and were subsequently improved
to $C^0$ by Le Calvez and Tal. The examples which we have presented
show that these results  are in some sense sharp.  Specifically, in
Theorem~\ref{thm:torus2}b), the examples with unbounded deviation,
based on Example~\ref{ex:not-goober}, have a unique supporting line;
and in Theorem~\ref{thm:torus2}c)iii), where there are lost measures,
at least in the cases we have been able to analyze,
the unique supporting line intersects the rotation set in a
non-trivial segment.
\end{enumerate}
\end{remarks}

\section{Questions raised by the family}
\label{sec:questions}
 The following properties hold for the
rotation sets~$\rho(t)$ of the family $\{\Phi_t\}$ constructed
here. Are they true in general? If not, are there natural conditions
under which they hold?

\begin{enumerate}[a)]
\item A point on the boundary of~$\rho(t)$ is a polygonal vertex if
  and only if it is a rational extreme point, and is a limit extreme
  point if and only if it is an irrational extreme point.
\item $\rho(t)$ has only finitely many irrational extreme points.
\item For every point $\bv\in\rho(t)$, including points on the
  boundary, there is a point $z\in\T^2$ with $\widehat{\delta^t}(z) =
  \bv$, so that $\rho_p(\Phi_t) = \rho_{\sMZ}(\Phi_t)$.
\item At least generically, totally irrational extreme points are
  smooth.
\item If $\rho(t_1) \not = \rho(t_2)$, then the function
  $t\mapsto\Ex(\rho(t))$ is discontinuous at some $t\in(t_1,t_2)$ (the
  {\em Tal-Zanata property}).
\end{enumerate}

\medskip
At least one of the properties of the family~$\{\Phi_t\}$ does not
hold in general, namely that every $\bv$ belonging to an interval with
rational endpoints contained in $\rho(t)$ is represented by a
\vgoober. This property does not hold in an example of Misiurewicz and
Ziemian (\cite{MZ2} Section~3).

\medskip

The proof of Theorem~\ref{thm:param-fam}, which was used to unwrap the
family of maps on the figure eight space, is essentially $C^0$.  Do
the phenomena observed here hold with more smoothness? What can one
say about generic rotation sets, and the rotation sets in generic
one-parameter families, in the $C^r$ category?

\bibliographystyle{amsplain}
\bibliography{torusrefs}

\end{document}